\documentclass[12pt]{amsart}
\usepackage{amsmath}
\usepackage{amsfonts}
\usepackage{amssymb}
\usepackage{amscd}
\usepackage{mathtools}
\usepackage{amsxtra}
\usepackage{amsthm}
\usepackage{graphicx}
\usepackage[abbrev,alphabetic]{amsrefs}
\RequirePackage[dvipsnames,usenames]{color}
\usepackage{soul,xcolor}
\setstcolor{red}
\usepackage{stmaryrd}
\usepackage{booktabs}
\usepackage{multirow}
\newtagform{tiny}{\tiny(}{)}

\usepackage{mathtools}
\usepackage{hyperref}
\usepackage[margin=1.25in]{geometry}
\usepackage{mathrsfs}

\usepackage{amsthm}
\usepackage{comment}
\usepackage[all,cmtip]{xy}
\usepackage{tikz-cd}
\usetikzlibrary{cd}

\tikzcdset{
  cells={font=\everymath\expandafter{\the\everymath\displaystyle}},
}

\usepackage[all]{xy}

\usepackage{cleveref}

\makeatletter
\def\@tocline#1#2#3#4#5#6#7{\relax
  \ifnum #1>\c@tocdepth 
  \else
    \par \addpenalty\@secpenalty\addvspace{#2}%
    \begingroup \hyphenpenalty\@M
    \@ifempty{#4}{%
      \@tempdima\csname r@tocindent\number#1\endcsname\relax
    }{%
      \@tempdima#4\relax
    }%
    \parindent\z@ \leftskip#3\relax \advance\leftskip\@tempdima\relax
    \rightskip\@pnumwidth plus4em \parfillskip-\@pnumwidth
    #5\leavevmode\hskip-\@tempdima
      \ifcase #1
       \or\or \hskip 1em \or \hskip 2em \else \hskip 3em \fi%
      #6\nobreak\relax
    \hfill\hbox to\@pnumwidth{\@tocpagenum{#7}}\par
    \nobreak
    \endgroup
  \fi}
\makeatother

\newcommand{\rup}[1]{\lceil #1 \rceil}
\newcommand{\rdown}[1]{\lfloor #1 \rfloor}

\renewcommand{\mod}{\ \textrm{mod}\ }
 
\renewcommand{\P}{\mathbb{P}}
\newcommand{\Z}{\mathbb{Z}}
\newcommand{\Q}{\mathbb{Q}}

\newcommand{\F}{\mathbb{F}}

\newcommand{\cO}{\mathcal{O}}

\newcommand{\m}{\mathfrak{m}}
\newcommand{\n}{\mathfrak{n}}

\newcommand{\fq}{\mathfrak{q}}

\newcommand{\perf}{\mathrm{perf}}

\newcommand{\wt}{\widetilde}

\def\var{\overline}

\DeclareMathOperator{\Spec}{Spec}

\DeclareMathOperator{\Hom}{Hom}

\DeclareMathOperator{\Tor}{Tor}
\DeclareMathOperator{\Ker}{Ker}

\DeclareMathOperator{\sht}{ht}

\DeclareMathOperator{\ppt}{ppt}

\renewcommand{\div}{{\rm div}}
\renewcommand{\Im}{\mathrm{Im}}

\theoremstyle{plain}
\newtheorem{theorem}{Theorem}[section]
\newtheorem{thm}[theorem]{Theorem}

\newtheorem{proposition}[theorem]{Proposition}

\newtheorem{lemma}[theorem]{Lemma}

\newtheorem{corollary}[theorem]{Corollary}

\newtheorem{claim}[theorem]{Claim}
\newtheorem*{claim*}{Claim}

\newtheorem{theoremA}{Theorem}

\theoremstyle{definition}
\newtheorem{definition}[theorem]{Definition}

\newtheorem{example}[theorem]{Example}
\newtheorem{notation}[theorem]{Notation}

\newtheorem*{setup*}{Setup}

\theoremstyle{remark}
\newtheorem{remark}[theorem]{Remark}
\newtheorem*{ackn}{Acknowledgements}

\theoremstyle{plain}

\newenvironment{claimproof}[0]
  {%
   \paragraph{\it Proof.}%
  }
  {%
    \hfill$\blacksquare$%
  }

\numberwithin{equation}{section}



\crefname{theorem}{Theorem}{Theorems}
\crefname{proposition}{Proposition}{Propositions}
\crefname{lemma}{Lemma}{Lemmas}
\crefname{corollary}{Corollary}{Corollaries}
\crefname{conjecture}{Conjecture}{Conjectures}
\crefname{claim}{Claim}{Claims}
\crefname{notation}{Notation}{Notations}
\crefname{remark}{Remark}{Remarks}
\crefname{example}{Example}{Examples}
\crefname{definition}{Definition}{Definitions}
\crefname{theoremA}{Theorem}{Theorems}

\title[Computation method for perfectoid purity and BCM-regularity]{Computation method for perfectoid purity and perfectoid BCM-regularity}
\author{Shou Yoshikawa}
\address{Institute of Science Tokyo, Tokyo 152-8551, Japan}
\email{yoshikawa.s.9fe9@m.isct.ac.jp}

\begin{document}

\begin{abstract}
In this paper, we introduce the notion of quasi-$F$-splitting for rings in mixed characteristic. By comparing quasi-$F$-splitting with perfectoid purity, we obtain a new inversion of adjunction-type result. Furthermore, we study the possible values of the perfectoid pure threshold of $\div(p)$ and construct new examples of perfectoid pure rings.
\end{abstract}



\maketitle

\tableofcontents

\setcounter{tocdepth}{1}

\section{Introduction}
The theory of singularities in mixed characteristic has recently seen significant development in numerous papers, such as \cite{MSTWW}, \cite{BMPSTWW24}, and \cite{HLS}. This theory has found applications in both algebraic geometry and commutative algebra in mixed characteristic, including the minimal model program \cite{BMPSTWW23}, \cite{TY23}, and Skoda-type theorems \cite{HLS}.

In this paper, we study two classes of singularities: \emph{perfectoid purity} \cite{p-pure} and \emph{perfectoid BCM-regularity} \cite{bcm-reg}, which are regarded as analogues of $F$-purity and strong $F$-regularity in positive characteristic, respectively. It is known that for a Noetherian local ring $(R, \m)$ with $p \in \m$, if $R$ is a complete intersection and $R/pR$ is $F$-pure or strongly $F$-regular, then $R$ is perfectoid pure or perfectoid BCM-regular, respectively.
On the other hand, a useful criterion known as \emph{Fedder's criterion} \cite{Fedder}, \cite{Glassbrenner} provides a practical method to test $F$-purity and strong $F$-regularity for complete intersection local rings. Many examples have been constructed using this criterion.

We now fix notation and recall the definition of perfectoid purity.
Let $(R, \m)$ be a $p$-torsion free Noetherian local ring with $p \in \m$ and such that $\var{R} := R/pR$ is $F$-finite. We say that $R$ is \emph{perfectoid pure} if there exists a pure extension $R \to B$ to a perfectoid ring.
If $R$ is perfectoid pure, the \emph{perfectoid pure threshold} $\ppt(R; \div(p))$ of $\div(p)$ is defined as
\[
\inf \left\{i/p^e \geq 0 \middle|
\begin{array}{l}
\text{there exist a perfectoid $R$-algebra $B$ a $p^e$-th root $\varpi$} \\
\text{of $p$, such that the map $R \to B \xrightarrow{\cdot \varpi^i} B$ is pure}
\end{array}
\right\},
\]
see \cref{test-perf-tests} for details. It is known \cite{p-pure}*{Theorem~6.6} that if $R$ is a complete intersection and $\var{R}$ is $F$-pure, then $R$ is perfectoid pure and $\ppt(R; \div(p)) = 1$ (cf.~\cref{remk:ppt}).

The first goal of this paper is to obtain a new inversion of adjunction-type result. To this end, we define the ring homomorphism
\[
\Phi_{R,n} \colon \var{R} \to \var{W}_n(R) := W_n(R)/pW_n(R), \quad a \mapsto (\wt{a}^p, 0, \ldots, 0),
\]
where $W_n(R)$ is the ring of Witt vectors over $R$ of length $n$, and $\wt{a} \in R$ is a lift of $a$ (see \cref{Phi-well-def} for well-definedness). We view $\var{W}_n(R)$ as an $\var{R}$-module via $\Phi_{R,n}$ and denote it by $Q_{R,n}$. Note that $Q_{R,1} = F_* \var{R}$.

We introduce the following notions:
\begin{definition}\label{def:Q:intro}\textup{(\cref{def:Q}, cf.~\cite{Yobuko19})}
\begin{itemize}
    \item Let $n$ be a positive integer. We say that $R$ is \emph{$n$-quasi-$F$-split} if $\Phi_{R,n} \colon \var{R} \to Q_{R,n}$ splits as $\var{R}$-modules. We say that $R$ is \emph{quasi-$F$-split} if it is $n$-quasi-$F$-split for some $n$.
    \item The \emph{quasi-$F$-splitting height} $\sht(R)$ is defined by
    \[
    \sht(R) := \inf\{n \geq 1 \mid \text{$R$ is $n$-quasi-$F$-split}\},
    \]
    and we set $\sht(R) := \infty$ if $R$ is not quasi-$F$-split.
    \item For positive integers $n$ and $e$, define the $\var{R}$-module $Q_{R,n,e}$ as the pushout in the following diagram:
    \[
    \begin{tikzcd}
        F^n_* \var{R} \arrow[r, "F^{e-1}"] \arrow[d, "V^{n-1}"] & F^{n+e-1}_* \var{R} \arrow[d] \\
        Q_{R,n} \arrow[r] & Q_{R,n,e} \arrow[ul, phantom, "\ulcorner", very near start]
    \end{tikzcd}
    \]
    The induced $\var{R}$-module homomorphism is denoted by
    \[
    \Phi_{R,n,e} \colon \var{R} \xrightarrow{\Phi_{R,n}} Q_{R,n} \to Q_{R,n,e}.
    \]
    \item We say that $R$ is \emph{quasi-$(F,F^{\infty})$-split} if $R$ is quasi-$F$-split and $\Phi_{R,\sht(R),e}$ splits as $\var{R}$-modules for every $e \geq 1$.
\end{itemize}
\end{definition}

Our first main result establishes a connection between quasi-$F$-splitting and perfectoid purity:

\begin{theoremA}\label{intro:ht-to-p-pure}
Assume $R$ is a complete intersection and let $n \geq 1$.
\begin{enumerate}
    \item $\sht(R) = n$ if and only if $R$ is perfectoid pure and
    \[
    1 - \frac{1}{p} - \cdots - \frac{1}{p^{n-1}} \geq \ppt(R; \div(p)) \geq 1 - \frac{p + \cdots + p^{n-1}}{p^n - 1}.
    \]
    In particular, if $R/pR$ is quasi-$F$-split, then $R$ is perfectoid pure.
    \item $R$ is quasi-$(F,F^\infty)$-split of height $n$ if and only if $R$ is perfectoid pure and
    \[
    \ppt(R; \div(p)) = 1 - \frac{1}{p} - \cdots - \frac{1}{p^{n-1}}.
    \]
\end{enumerate}
\end{theoremA}

We also establish a Fedder-type criterion for quasi-$F$-splitting (\cref{Fedder}) and provide explicit examples, including cases where $\ppt(R; \div(p)) \neq 1$ (cf.~\cref{ex:threshold}).

\begin{enumerate}
    \item Let $R:=\Z_{(p)}[[x,y,z]]/(z^2+y^3+z^5)$, then
    \[
    \ppt(R;\div(p))=
    \begin{cases}
        1/8 & \textup{if $p=2$}, \\
        5/9 & \textup{if $p=3$}, \\
        4/5 & \textup{if $p=5$}. 
    \end{cases}
    \]
    \item Let $R:=\Z_{(p)}[[x,y,z,w]]/(w^2+xyz(x+y+z))$ and $p=2$, then $\ppt(R;\div(p))=1/2$.
    We note that $\var{R}$ is not quasi-$F$-split by \cite{kty}*{Example~7.11}.
    \item Let $R:=\Z_{(p)}[[x,y,z,w]]/(w^2+xyz(x+y+z)+p(xy+xz+yz)w)$ and $p=2$, then $\ppt(R;\div(p))=1/4$.
    \item Let $R:=\Z_{(p)}[[x,y,z]]/(z^2+x^2y+xy^n)$ for $n \geq 2$ and $p=2$, then $\ppt(R;\div(p))=1/2^{\rup{\log_2 n}}$.
\end{enumerate}
In addition, we prove a result on the possible values of the perfectoid pure threshold:
\begin{theoremA}\label{intro:range-ppt}\textup{(\cref{range-ppt})}
We assume $R$ is complete intersection and perfectoid pure.
If 
\[
\ppt(R;\div(p)) > \frac{p-2}{p-1},
\]
then $\ppt(R;\div(p))$ satisfies
\[
1-\frac{1}{p}-\cdots-\frac{1}{p^{n-1}} \geq \ppt(R;\div(p)) \geq 1- \frac{p+\cdots+p^{n-1}}{p^n-1}
\]    
for some positive integer $n$.
\end{theoremA}

The next main result \cref{intro:Fano} concerns the  perfectoid BCM-regularity (\cite{bcm-reg}*{Definition~6.9}) for graded rings.
\begin{theoremA}\label{intro:Fano}\textup{(\cref{Fano})}
Let $S=\bigoplus_{m \geq 0} S_m$ be a Noetherian graded ring such that $S$ is $p$-torsion free and $S_0$ is a divisorial valuation ring with $p$ in the maximal ideal of $S_0$.
We assume $\var{S}:=S/pS$ is normal quasi-Gorenstein and 
\[
a(\var{S}):=\max \{m \in \Z \mid H^{\dim{\var{S}}}_\m(\var{S})_m \neq 0\} <0.
\]
We further assume that $\Spec{\var{S}} \backslash \{\var{S}_{>0}\}$ is strongly $F$-regular.
If $S_{(p,S_{>0})}$ is quasi-$F$-split, then $S_{(p,S_{>0})}$ is perfectoid BCM-regular.
Furthermore, the converse implication holds if $p=2$.
\end{theoremA}
\noindent
Moreover, we establish a Fedder-type criterion for quasi-$F$-splitting (\cref{Fedder}), which in particular serves as a criterion for perfectoid BCM-regularity for rings satisfying the assumptions of \cref{intro:Fano} in characteristic two.
Furthermore, an inversion of adjunction-type result (\cref{compsre-posi-case}) shows that if $\var{S}$ is quasi-$F$-split, then so is $S$.

We now present applications of the above theorems by constructing explicit examples.
\begin{theoremA}\label{intro:example-new}\textup{(\cref{example-new}, cf.~\cref{log-fano,log-fano-pair})}
There exists a Noetherian graded ring $S=\bigoplus_{i \in \Z_{\geq 0}}S_i$ such that $S$ satisfies the following properties.
\begin{enumerate}
    \item $S$ is torsion free and $S_0$ is a divisorially valuation ring with maximal ideal $(p)$.
    \item We set $\m:=(S_{>0},p)$.
    Then $S_\m$ is perfectoid BCM-regular.
    \item $S_\m/pS_\m$ is normal Gorenstein, but not strongly $F$-regular.
\end{enumerate}
\end{theoremA}
\noindent
We also compute the perfectoid pure threshold for quasi-$F$-split graded rings with $a(\var{S}) = 0$.
\begin{theoremA}\label{intro:CY-case}\textup{(\cref{CY-case-qfs})}
Let $S=\bigoplus_{m \geq 0} S_m$ be a Noetherian graded ring such that $S$ is $p$-torsion free and $S_0$ is a divisorial valuation ring with $p$ in the maximal ideal of $S_0$.
We assume $S_\m$ is complete intersection and 
\[
a(\var{S}):=\max \{m \in \Z \mid H^{\dim{\var{S}}}_\m(\var{S})_m \neq 0\} =0,
\]
where $\m:=(S_{>0},p)$.
If $\sht(\var{S_\m})=n < \infty$, then $S_\m$ is perfectoid pure and we have
\[
\ppt(S_\m;\div(p))=1-\frac{p+\cdots+p^{n-1}}{p^n-1}.
\]
\end{theoremA}
\noindent
Finally, we present an example of a ring that is perfectoid pure, but whose self-tensor product is not.
\begin{example}\textup{(\cref{fermat-ell})}
Let $R':=\mathbb{Z}_{p}[[x,y,z]]/(x^3+y^3+z^3)$ and $p \equiv 2 \mod 3$.
Then $R'$ is perfectoid pure but $R' \otimes_{\Z_p} R'$ is not perfectoid pure.
\end{example}

\begin{ackn}
The author wishes to express his gratitude to Ryo Ishizuka, Teppei Takamastu, Kenta Sato, and Tatsuro Kawakami for valuable discussion.
He is also grateful to Linquan Ma for helpful comments. 
The author was supported by JSPS KAKENHI Grant number JP24K16889.
\end{ackn}

\section{Preliminary}

\subsection{Notation}\label{ss-notation}
In this subsection, we summarize the notation and terminology used throughout this paper.

\begin{enumerate}
\item We fix a prime number $p$ and set $\F_p := \Z/p\Z$.  
\item For a ring $R$ of characteristic $p>0$, we denote by $F \colon R \to R$ the absolute Frobenius ring homomorphism, defined by $F(r) = r^p$ for all $r \in R$.

\item For a ring $R$ of characteristic $p>0$, we say that $R$ is {\em $F$-finite} if $F$ is finite. 
\item For a ring $R$ of characteristic $p>0$, we define the \emph{perfection} of $R$ as $R_{\mathrm{perf}} := \mathrm{colim}_F R$.
\item We use the terminology \emph{perfectoid} as defined in \cite{BMS19}*{Definition~3.5}. For a semiperfectoid ring $R$, its \emph{perfectoidization} is denoted by $R_{\mathrm{perfd}}$ (cf.~\cite{BS}*{Section~7}).
\item Let $R$ be a ring and $f$ an element of $R$.
A sequence $\{f^{1/p^e}\}_{e \geq 0}$ is called \emph{a compatible system of $p$-power roots of $f$} if $(f^{1/p^{e+1}})^p=f^{1/p^e}$ for all $e \geq 0$ and $f^{1/p^0}=f$.
Furthermore, the ideal $(f^{1/p^e} \mid e \geq 0)$ of $R$ is denoted by $(f^{1/p^{\infty}})$.
\item For a ring $R$ and a ring homomorphism $\phi \colon R \to R$, we say that $\phi$ is a \emph{lift of Frobenius} if $\phi(a) \equiv a^p \mod p$ for every $a \in R$.
\item For a ring homomorphism $f \colon A \to B$ and a finitely generated ideal $I$ of $A$, we say that $f$ is \emph{$I$-completely faithfully flat} if for every $I$-torsion $A$-module $N$, we have $\Tor_i^A(B,N)=0$ for every $i \neq 0$ and $B/IB$ is faithfully flat over $A/I$.
If $I=(p)$, we say that $f$ is $p$-completely faithfully flat.
\item For a ring homomorphism $f \colon A \to B$, we say that $f$ is \emph{pure} if for every $A$-module $M$, the homomorphism $M \to M \otimes_A B$ is injective. 
\item Let $(R,\m)$ be a local ring and $B$ an $R$-algebra.
We say that $B$ is \emph{big Cohen-Macaulay (BCM) algebra over $R$} if for every system of parameter $x_1,\ldots,x_d$ of $R$, the sequence $x_1,\ldots,x_d$ is a regular sequence on $B$.
Furthermore, we say that $R$ is \emph{perfectoid BCM-regular} if every extension $R \to B'$ to a perfectoid BCM $R$-algebra $B'$ is pure (\cite{bcm-reg}*{Definition~6.9}).
\end{enumerate}

\subsection{The ring of Witt vectors}\label{subsec:Witt vectors}
We recall the definition of the ring of Witt vectors. The basic reference is \cite{Serre}*{II,~\S~6}.

\begin{lemma}
\label{lem:witt polynomial}
For $n\in \Z_{\geq 0}$, we define the polynomial $\varphi_{n} \in \Z[X_{0},\ldots X_{n}]$ by
\[
\varphi_{n} := \sum_{i=0}^{i=n} p^i X_{i}^{p^{n-i}}.
\]
Then, there exist polynomials $S_{n} ,P_{n} \in \Z[X_{0},\ldots, X_{n}, Y_{0}, \ldots, Y_{n}]$ for any $n \in \Z_{\geq 0}$ such that
\[
\varphi_{m} (S_{0}, \ldots, S_{m}) = \varphi_{m}(X_{0}, \ldots X_{m}) +\varphi_{m}(Y_{0}, \ldots Y_{m})
\]
and 
\[
\varphi_{m} (P_{0}, \ldots, P_{m}) = \varphi_{m}(X_{0}, \ldots X_{m}) \cdot \varphi_{m}(Y_{0}, \ldots Y_{m})
\]
hold for any $m \in \Z_{\geq 0}$.
\end{lemma}

\begin{proof}
See \cite{Serre}*{II, Theorem 6}.
\end{proof}

\begin{definition}
\label{defn:Witt ring}
Let $A$ be a ring.
\begin{itemize}
\item We set
\[
W(A) := \{
(a_{0}, \ldots, a_{n}, \ldots) | a_{n} \in A
\} = \prod_{\Z_{\geq 0}} A,
\]
equipped with the addition
\[
(a_{0}, \ldots, a_{n}, \ldots) + (b_{0}, \ldots, b_{n}, \ldots) := (S_{0}(a_{0},b_{0}), \ldots, S_{n}(a_{0}, \ldots, a_{n}, b_{0}, \ldots, b_{n}), \ldots)
\]
and the multiplication
\[
(a_{0}, \ldots, a_{n}, \ldots) \cdot (b_{0}, \ldots, b_{n}, \ldots) := (P_{0}(a_{0},b_{0}), \ldots, P_{n}(a_{0}, \ldots, a_{n}, b_{0}, \ldots, b_{n}), \ldots).
\]
Then $W(A)$ is a ring with $1= (1, 0, \ldots) \in W(A)$, which is called the \emph{ring of Witt vectors} over $A$.
For $a \in A$, we set $[a]:=(a,0,\ldots ) \in W(A)$.
We define an additive map $V$ by
\[
V \colon W(A) \rightarrow W(A); (a_{0}, a_{1}, \ldots) \mapsto (0,a_{0}, a_{1}, \ldots).
\]
\item We define the ring $W_{n}(A)$ of \emph{Witt vectors over $A$ of length $n$} by
\[
W_{n}(A) := W(A)/ V^{n}W(A).
\]
We note that $V^{n}W(A)\subseteq W(A)$ is an ideal.
For $a \in A$ and $n \in \Z_{\geq 1}$, we set $[a]:=(a,0,\ldots ,0) \in W_n(A)$.
We define a $W(A)$-module homomorphism $R$ by
\[
W(A) \rightarrow W_{n}(A) ; (a_{0}, a_{1}, \ldots) \mapsto (a_{0}, a_{1}, \ldots, a_{n-1})
\]
and call $R$ the restriction map.
Moreover, we define 
\begin{eqnarray*}
V \colon W_{n}(A) \rightarrow W_{n+1}(A),
\end{eqnarray*}
and
\[
R \colon W_{n+1} (A) \rightarrow W_{n}(A),
\]
in a similar way.
The image of $\alpha \in W(A)$ or $W_n(A)$ by $V$ is denoted by $V\alpha$.
\item For any ring homomorphism $f \colon A \to A'$, a ring homomorphism
\[
W(A) \to W(A')\ ;\ (a_0,a_1,\ldots) \mapsto (f(a_0),f(a_1),\ldots)
\]
is induced and the induced homomorphism is commutative with $R$ and $V$.
Note that $f$ also induces a ring homomorphism $W_n(A) \to W_n(A')$.
\item For an ideal $I$ of $A$,
we denote the kernels of $W_n(A) \to W_n(A/I)$ and $W_n(A) \to W_n(A/I)$ induced by the natural quotient map $A \to A/I$ are denoted by $W(I)$ and $W_n(I)$, respectively.
\item We define a ring homomorphism $\varphi$ by
\[
\varphi \colon W(A) \rightarrow \prod_{\Z_{\geq 0}} A; (a_{0}, \ldots, a_{n}, \ldots) \mapsto (\varphi_{0}(a_{0}), \ldots, \varphi_{n}(a_{0}, \ldots, a_{n}), \ldots), 
\]
where $\prod_{\Z_{\geq 0}} A$ is the product in the category of rings.
The map $\varphi$ is injective if $p \in A$ is a non-zero divisor.
The element $\varphi_{n}(a_{0}, \ldots, a_{n})$ (resp.\,the element $\varphi(a_{0}, \ldots a_{n}, \ldots)$) is called the ghost component (resp.\,the vector of ghost components) of $(a_{0}, \ldots a_{n}, \ldots) \in W(A)$.
\end{itemize}
\end{definition}

\begin{proposition}\label{rule-witt}
Let $R$ be a ring.
Then we have the following properties:
\begin{enumerate}
    \item For $a,b \in R$ $m \in \Z_{\geq 0}$, we have $[a]V^m([b])=V^m([a^{p^m}b])$ in $W(R)$ .
    \item For $\alpha,\beta \in W(R)$, we have $V(\alpha)V(\beta)=pV(\alpha\beta)$ in $W(R)$.
\end{enumerate}
\end{proposition}

\begin{proof}
We take a ring homomorphism $\pi \colon A:=\Z[X_f \mid f \in R] \to R$ such that $\pi(X_f)=f$, then $\pi$ is surjective.
Since $\pi$ induces a ring homomorphism  $W(\pi) \colon W(A) \to W(R)$ and $A$ is $p$-torsion free, we may assume that $R$ is $p$-torsion free.
We take $\varphi \colon W(R) \to \prod_{\Z_{\geq 0}}R$  as in \cref{defn:Witt ring}, then $\varphi$ is injective.
Let $s \colon \prod_{\Z_{\geq 0}}R \to \prod_{\Z_{\geq 0}}R$ be a map defined by $s(a_0,a_1,\ldots)=(0,a_0,a_1,\ldots)$, then we have $s(\alpha\beta)=s(\alpha)s(\beta)$ for $\alpha,\beta \in \prod_{\Z_{\geq 0}}R$.
We prove assertion (1).
We have
\begin{align*}
    \varphi([a]V^m([b]))&=\varphi([a])\varphi(V^m([b]))=(a,a^p,\cdots)p^ms^m(b,b^p,\ldots) \\
    &=p^ms^m(a^{p^m}b,a^{p^{m+1}}b^p,\ldots)=\varphi(V^m([a^{p^m}b])),
\end{align*}
for $a,b \in R$.
Since $\varphi$ is injective, we obtain assertion (1).
Next, we prove assertion (2).
Furthermore, we have $\varphi \circ V(\alpha)=ps \circ \varphi(\alpha)$ for $\alpha \in W(R)$, thus
\[
\varphi(V(\alpha) V(\beta))=\varphi\circ V(\alpha)\varphi\circ V(\beta)=p^2s\circ \varphi(\alpha)s \circ \varphi(\beta)=p^2s\circ \varphi(\alpha\beta)=pV(\alpha\beta)
\]
for $\alpha, \beta \in W(R)$, as desired.
\end{proof}

\subsection{Test perfectoid and perfectoid pure threshold}
In this subsection, we introduce the notion of test perfectoid and  perfectoid pure threshold.
The existence of the test perfectoids is proved in Appendix \ref{app:A}.

\begin{definition}\textup{(cf.~\cref{thm:const-test-perfd})}
Let $R$ be a ring and $R_{\infty}$ a perfectoid $R$-algebra.
We say that $R_{\infty}$ is a \emph{test perfectoid over $R$} if for every perfectoid $R$-algebra $B$, there exist a $p$-completely faithfully flat $R$-algebra homomorphism $B \to B'$ to a perfectoid $R$-algebra $B'$ and an $R$-algebra homomorphism $R_{\infty} \to B'$.
\end{definition}

\begin{remark}\label{p-ff-bound}
Let $R \to S$ be a ring homomorphism of perfectoids.
Then $R \to S$ is $p$-completely faithfully flat if and only if $R/p^n \to S/p^n$ is faithfully flat for each positive integer $n$ by \cite{BS}*{Lemma~3.8(2)} and \cite{BMS19}*{Corollary~4.8(3)}.    
\end{remark}

\begin{proposition}\label{test-perf-perfectoidazation}
Let $R$ be a ring, and $R_{\infty}$ a perfectoid over $R$.
\begin{enumerate}
    \item Let $R_\infty \to R'_\infty$ be a $p$-completely faithfully flat $R$-algebra homomorphism to a perfectoid $R'_\infty$.
    Then $R_\infty'$ is a test perfectoid over $R$.
    \item Let $I$ be an ideal of $R$.
    Then $(R_{\infty}/IR_{\infty})_{perfd}$ is a test perfectoid over $R/I$.
\end{enumerate}
\end{proposition}

\begin{proof}
We prove (1).
We take a perfectoid $R$-algebra $B$, then replacing $B$ by a $p$-completely faithfully flat extension of $B$, we may assume that $B$ is $R_\infty$-algebra.
By \cite{CS}*{Proposition~2.1.11(b)}, the $p$-adic completion $V':=B \widehat{\otimes}_{R_\infty} R'_\infty$ of $B \otimes_{R_\infty} R'_\infty$ is a perfectoid.
Since $R_\infty \to R'_\infty$ is $p$-completely faithfully flat, so is $B \to B'$ by \cref{p-ff-bound}.
Since there exists an $R$-algebra homomorphism $R'_\infty \to B'$, as desired.
Next, we prove (2).
We take a perfectoid $R/I$-algebra $B$.
Since $R_{\infty}$ is a test perfectoid, there exists a $p$-completely faithfully flat extension $B \to B'$ of perfectoid $R$-algebras such that there exists an $R$-algebra homomorphism $R_\infty \to B'$.
Since $B'$ is an $R/I$-algebra and $R_{\infty}$-algebra, it is an $R_{\infty}/IR_{\infty}$-algebra.
Thus, we obtain the $R/I$-algebra homomorphism $(R_{\infty}/IR_{\infty})_{perfd} \to B'$, as desired.
\end{proof}

\begin{lemma}\label{positive-chara-test}
Let $R$ be an $\F_p$-algebra, $R_{\infty}$ a test perfectoid over $R$.
Let $M$ be an $R$-module.
Then we have
\[
\mathrm{Ker}(M \to R_{\infty} \otimes_R M )=\mathrm{Ker}(M \to R_{perf} \otimes_R M),
\]
where $R_{perf}$ is the perfection of $R$.
\end{lemma}

\begin{proof}
Since $R_{\infty}$ is a perfectoid $\F_{p}$-algebra, it is a perfect ring.
Therefore, there exists an $R$-algebra homomorphism $R_{\perf} \to R_{\infty}$, thus we have
\[
\mathrm{Ker}(M \to R_{perf} \otimes_R M) \subseteq \mathrm{Ker}(M \to R_{\infty} \otimes_R M ).
\]
On the other hand, since $R_{\infty}$ is a test perfectoid, there exist a faithfully flat extension $R_{\perf} \to B$ to a perfect $R$-algebra and an $R$-algebra homomorphism $R_{\infty} \to B$.
Therefore we have
\[
\mathrm{Ker}(M \to R_{\infty} \otimes_R M ) \subseteq \mathrm{Ker}(M \to B \otimes_R M) \overset{(*)}{=}\mathrm{Ker}(M \to R_{perf} \otimes_R M),
\]
where $(*)$ follows from the faithfully flatness of $R_{\perf} \to B$.
\end{proof}

\begin{definition}\label{def:ppt}
Let $R$ be a $p$-torsion free Noetherian ring with $p$ in its Jacobson radical.
We assume $R$ is perfectoid pure.
We define the \emph{perfectoid pure threshold} $\ppt(R;\div(p))$ of $\div(p)$  by
\[
\inf \left\{i/p^e \geq 0 \middle| 
\begin{array}{l}
\textup{there exist a perfectoid $R$-algebra $B$ and $\varpi$ a $p^e$-power} \\
\textup{root of $p$ such that $R \to B \xrightarrow{\cdot \varpi^i} B$ is pure}
 \end{array}
 \right\}. 
\]
\end{definition}

\begin{proposition}\label{test-perf-tests}
Let $R$ be a $p$-torsion free Noetherian ring with $p$ in its Jacobson radical.
Then we obtain the following assertions.
\begin{enumerate}
    \item The ring $R$ is perfectoid pure if and only if every test perfectoid $R_{\infty}$ over $R$  is pure.
    \item Let $R_\infty$ be a test perfectoid over $R$  and $\{p^{1/p^e}\}$ a compatible system of $p$-power roots of $p$ in $R_\infty$.
    If $R$ is perfectoid pure and $R_\infty$ is $p$-torsion free, then we have 
    \[
    \ppt(R;\div(p))=\inf \{\alpha \in \Z[1/p]_{\geq 0} \mid \textup{$R \to R_{\infty} \xrightarrow{\cdot p^{\alpha}} R_\infty$ is pure } \}.
    \]
    \item We assume $(R,\m)$ is a complete local ring. 
    We take $S \to R$ and $R^{S_\infty}_{perfd}$ be as in \cite{p-pure}*{Lemma~4.23}.
    Then we have
    \[
    \ppt(R;\div(p))=\inf \{\alpha \in \Z[1/p]_{\geq 0} \mid \textup{$R \to R^{S_\infty}_{perfd} \xrightarrow{\cdot p^{\alpha}} R^{S_\infty}_{perfd}$ is pure } \}.
    \]
\end{enumerate}
\end{proposition}

\begin{proof}
We assume $R$ is perfectoid pure and take a test perfectoid $R \to R_{\infty}$ over $R$.
Since $R$ is perfectoid pure, there exists a pure extension $R \to B$ such that $B$ is a perfectoid.
Since $R_{\infty}$ is a test perfectoid, we may assume that there exists an $R$-algebra homomorphism $R_{\infty} \to B$ by the proof of \cite{p-pure}*{Lemma~4.5}. 
Therefore, $R \to R_{\infty}$ is also pure.
On the other hand, by \cref{thm:const-test-perfd}, $R$ has a test perfectoid, thus we obtain the converse implication.

Next, we prove assertion (2).
By definition, we have
\[
\inf \{\alpha \in \Z[1/p]_{\geq 0} \mid \textup{$R \to R_{\infty} \xrightarrow{\cdot p^{\alpha}} R_\infty$ is pure } \} \geq \ppt(R;\div(p)).
\]
We take $i,e \in \Z_{\geq 1}$, a perfectoid $R$-algebra $B$, and $p^e$-th root $\varpi$ of $p$ such that $R \to B \xrightarrow{\cdot \varpi^i} B$ is pure.
It is enough to show that $R \to R_\infty \xrightarrow{\cdot p^{i/p^e}} R_\infty$ is pure.
By the proof of (1) and the definition of test perfectoid, we may assume there exists an $R$-algebra homomorphism $\nu \colon R_\infty \to B$.
By \cite{BMS18}*{Lemma~3.10}, we have $\varpi B=\nu(p^{1/p^e})B$, thus there exists a unit $u \in B$ such that $\varpi=u\nu(p^{1/p^e})$, and in particular, we have $R \to B \xrightarrow{\cdot \nu(p^{1/p^e})} B$ is pure.
By the commutative diagram
\[
\begin{tikzcd}
    R \arrow[r] & R_\infty \arrow[r,"\cdot p^{1/p^e}"] \arrow[d] & R_\infty \arrow[d] \\
    & B \arrow[r," \cdot \nu(p^{1/p^e})"] & B,
\end{tikzcd}
\]
we obtain that $R \to R_\infty \xrightarrow{\cdot p^{1/p^e}} R_\infty$ is pure.

Finally, assertion (3) follows from the same argument as in the proof of assertion (2) and the proof of \cite{p-pure}*{Lemma~4.23}.
\end{proof}

\begin{remark}\label{remk:ppt}
Let $(R,\m)$ be a $p$-torsion free Noetherian local ring with $p \in \m$.
Then $\ppt(R;\div(p))=1$ if and only if $(R,p)$ is perfectoid pure by means of \cite{p-pure}*{Definition~6.1}.
Indeed, we take a test perfectoid $R_\infty$ over $R$ such that $R_\infty$ contains a compatible system $\{p^{1/p^e}\}$ of $p$-power root of $p$ and $R_\infty$ is $p$-torsion free, where the existence follows from \cref{thm:const-test-perfd}.
Since $R$ and $R_\infty$ are $p$-torsion free, the pair $(R,p)$ is perfectoid pure if and only if $R \to R_\infty \xrightarrow{\cdot p^{1-1/p^e}} R_\infty$ is pure for every $e \geq 1$, thus it is equivalent to $\ppt(R;\div(p))=1$ by \cref{test-perf-tests}.
In particular, if $R/pR$ is $F$-pure and $R$ is complete intersection, then $\ppt(R;\div(p))=1$ by \cite{p-pure}*{Theorem~6.6}.  
\end{remark}

\section{Witt vectors of rings with Frobenius lifts}
In this section, we study the structure of the ring of Witt vectors of rings with Frobenius lift.
\cref{decomposition-W} is a key ingredient of results in the paper.

\begin{notation}\label{notation:F-lift}
Let $A$ be a $p$-torsion free ring with a ring homomorphism $\phi \colon A \to A$ such that $\phi(a) \equiv a^p \mod p$ for every $a \in A$.
The ring homomorphism
\[
\varphi:=\prod \varphi_r \colon W(A) \to \prod_{\mathbb{N}} A
\]
is defined as in \cref{defn:Witt ring}. 
\end{notation}

\begin{proposition}\label{s_phi}\textup{(cf.~\cite{Ill79}*{Section~1.3})}
We use the notation introduced in \cref{notation:F-lift}.
There exists a ring homomorphism
\[
s_\phi \colon A \to W(A)
\]
such that 
\[
\varphi_n(s_\phi(a))=\phi^n(a).
\]
The composition $A \xrightarrow{s_\phi} W(A) \to W_n(A)$ is also denoted by $s_\phi$.
\end{proposition}

\begin{proof}
It is enough to show that for every $a \in A$, there exist $a_0,a_1,\ldots$ such that
\[
\sum_{0 \leq s \leq n} p^sa_s^{p^{n-s}}=\phi^n(a).
\]
Then if we define $s_\phi(a)=(a_0,a_1,\ldots)$, we have $\varphi_n(s_\phi(a))=\phi^n(a)$.
Since $\varphi$ is injective, $s_\phi$ is a ring homomorphism.

Therefore, we prove the existence of $a_0,a_1,\ldots$ as above by the induction on $n$.
First, we define $a_0:=a$.
Next, we assume $n \geq 1$ and there exist $a_0,a_1,\ldots,a_{n-1}$ such that
\[
\sum_{0 \leq s \leq n-1} p^sa_s^{p^{n-s-1}}=\phi^{n-1}(a).
\]
Then we have
\[
\phi^n(a)=\sum_{0 \leq s \leq n-1} p^s\phi(a_s)^{p^{n-s-1}}
\]
Since $\phi(a_s) \equiv a_s^p \mod p$, we have $p^s\phi(a_s)^{p^{n-s-1}} \equiv p^sa_s^{p^{n-s}} \mod p^{n}$.
Therefore, we have
\[
\phi^n(a) \equiv \sum_{0 \leq s \leq n-1} p^sa_s^{p^{n-s}} \mod p^n,
\]
and there exists $a_{n}$ such that
\[
\phi^n(a)=\sum_{0 \leq s \leq n} p^sa_s^{p^{n-s}},
\]
as desired.
\end{proof}

\begin{proposition}\label{V-module-hom}
We use the notation introduced in \cref{notation:F-lift}.
We regard $W_n(A)$ as $A$-module by $s_\phi$ for every positive integer $n$.
Then $V \colon \phi_*W_{n-1}(A) \to W_n(A)$ is an $A$-module homomorphism for $n \geq 2$.
\end{proposition}

\begin{proof}
First, we note that the diagram
\[
\begin{tikzcd}
    \phi_*W_{n-1}(A) \arrow[r,"V"] \arrow[d,"\phi_*\varphi"] & W_n(A) \arrow[d,"\varphi"] \\
    \phi_*A \oplus \cdots \oplus \phi^n_*A \arrow[r,"V'"] & A \oplus \cdots \oplus \phi^n_*A
\end{tikzcd}
\]
commutes, where $V'$ is defined by $V'((a_0,\ldots,a_{n-1}))=(0,pa_0,\ldots,pa_{n-1})$.
Furthermore, since $V'$ is an $A$-module homomorphism and the $A$-module homomorphisms
$\phi_*\varphi$,  $\varphi$ are injective, the map $V$ is $A$-module homomorphism, as desired.
\end{proof}

\begin{definition}
We use the notation introduced in \cref{notation:F-lift}.
We define a map
\[
\Delta_{W_n} \colon W_n(A) \to W_{n-1}(A)
\]
by
\[
V\Delta_{W_n}((a_0,\ldots,a_{n-1}))=(a_0,\ldots,a_{n-1})-s_\phi(a_0).
\]
Since $s_\phi(a_0)-[a_0]$ is contained in the image of $V$, the map $\Delta_{W_n}$ is well-defined.
\end{definition}

\begin{proposition}\label{prop:delta map'}
We use the notation introduced in \cref{notation:F-lift}.
Maps $\Delta_{W_n}$ are group homomorphisms and satisfy the following properties
\begin{enumerate}
    \item $\Delta_{W_n}(s_\phi(a))=0$ for $a \in A$ , 
    \item $\Delta_{W_n}(V\alpha)=\alpha$ for $\alpha \in W_{n-1}(A)$, and
    \item $\Delta_{W_n}(\alpha\beta) \equiv s_\phi \circ \phi(b)\Delta_{W_n}(\alpha)+s_\phi \circ \phi(a)\Delta_{W_n}(\beta) \mod pW_{n-1}(A)$ for $\alpha,\beta \in W_n(A)$ and $a:=R^{n-1}(\alpha),b:=R^{n-1}(\beta)$.
\end{enumerate}
\end{proposition}

\begin{proof}
We take an element $a \in A$, then 
\[
V\Delta_{W_n}(s_\phi(a))=s_\phi(a)-s_\phi(a)=0,
\]
thus we have $\Delta_{W_n}(s_\phi(a))=0$.
The second assertion is clear by definition of $\Delta_{W_n}$.
For the third condition, we consider
\begin{align*}
    V\Delta_{W_n}(\alpha\beta) &=\alpha\beta-s_\phi(ab) \\
    &=(\alpha-s_\phi(a))\beta+s_\phi(a)(\beta-s_\phi(b)) \\
    &=s_\phi(b)V\Delta_{W_n}(\alpha)+s_\phi(a)V\Delta_{W_n}(\beta)+(V\Delta_{W_n}(\alpha))(V\Delta_{W_n}(\beta)) \\
    &\overset{(\star)}{=}V(s_\phi \circ \phi(b)\Delta_{W_n}(\alpha)+s_\phi \circ \phi(a)\Delta_{W_n}(\beta)+p\Delta_{W_n}(\alpha)\Delta_{W_n}(\beta)),
\end{align*}
where $(\star)$ follows from \cref{V-module-hom} and \cref{rule-witt} (2).
Therefore, we have
\[
\Delta_{W_n}(\alpha\beta)=s_\phi \circ \phi(b)\Delta_{W_n}(\alpha)+s_\phi \circ \phi(a)\Delta_{W_n}(\beta)+p\Delta_{W_n}(\alpha)\Delta_{W_n}(\beta),
\]
as desired.
\end{proof}

\begin{thm}\label{thm:delta formula}\textup{(c.f.~\cite{kty}*{Lemma 3.17})}
We use the notation introduced in \cref{notation:F-lift}.
Let $a \in A$.
We define the map
\[
\Delta_s(a)=\Delta_{W_2} \circ \Delta_{W_3} \circ \cdots \circ \Delta_{W_{s+1}}([a]) \in A
\]
for $s \geq 1$, and $\Delta_0(a)=a$.
Then for every $a \in A$, we have 
\[
[a]=s_\phi(a)+Vs_\phi(\Delta_1(a))+\cdots + V^{n-1}s_\phi(\Delta_{n-1}(a)) +V^{n}\Delta_{n}(a).
\]
In particular, we have the following two equations
\begin{align*}
    a^{p^{n}}&=\sum^n_{s=0}p^s\phi^{n-s}\Delta_s(a) \\
    \Delta_n(a)&=\frac{\Delta_1(a^{p^{n-1}})}{p^{n-1}}. 
\end{align*}
Furthermore, we obtain
\[
\Delta_n(a) \equiv
\begin{cases}
    a^{p^n-p}\Delta_1(a)+a^{p^n-2p}\Delta_1(a)^p \mod p & p=2 \\
    a^{p^n-p}\Delta_1(a) \mod p & \textup{otherwise}
\end{cases}
\]
for integers $n \geq 2$.
\end{thm}

\begin{proof}
By the definition of $\Delta_{W_{n+1}}([a])$, we have
\[
[a]=s_\phi(a)+V\Delta_{W_{n+1}}([a]).
\]
Furthermore, since the first component of $\Delta_{W_{n+1}}([a])$ is $\Delta_1(a)$, we have
\[
\Delta_{W_{n+1}}([a])=s_\phi(\Delta_1(a))+V\Delta_{W_{n}}(\Delta_{W_{n+1}}([a])).
\]
We note that the first component of $\Delta_{n-r} \circ \cdots \Delta_{n+1}([a])$ is $\Delta_{r}(a)$, thus, repeating such a process, we have
\[
[a]=s_\phi(a)+Vs_\phi(\Delta_1(a))+\cdots + V^{n-1}s_\phi(\Delta_{n-1}(a)) +V^{n}\Delta_{n}(a).
\]
Furthermore, considering the ghost component $\varphi_n$, we have
\begin{equation}\label{eq:delta}
    a^{p^n}=\phi^n(a)+p\phi^{n-1}\Delta_1(a)+\cdots +p^{n}\Delta_{n}(a)=\sum^n_{s=0} p^s\phi^{n-s}\Delta_s(a).
\end{equation}
Moreover, we have
\begin{align*}
    p\Delta_1(a^{p^{n-1}})
    &=a^{p^{n}}-\phi(a^{p^{n-1}}) \\
    &\overset{(\star_1)}{=} \sum^n_{s=0} p^s\phi^{n-s}\Delta_s(a)-\sum^{n-1}_{s=0}p^s\phi^{n-s}\Delta_s(a) \\
    &=p^n\Delta_n(a),
\end{align*}
where $(\star_1)$ follows from (\ref{eq:delta}).
The last assertion follows from the proof of \cite{kty}*{Theorem~3.20}.
\end{proof}

\begin{theorem}\label{decomposition-W}
We use the notation introduced in \cref{notation:F-lift}.
We regard $W_n(A)$ as $A$-module by $s_\phi$.
Then we obtain an $A$-module isomorphism
\[
\Psi_n \colon W_n(A) \xrightarrow{\sim} A \oplus \phi_*A \oplus \cdots \oplus \phi^{n-1}_*A,
\]
where $\Psi_n$ is defined by
\[
\Psi_n((a_0,\ldots,a_{n-1}))=(a_0,\Delta_1(a_0)+a_1,\ldots,\Delta_{n-1}(a_0)+\Delta_{n-2}(a_1)+\cdots+\Delta_1(a_{n-2})+a_{n-1}).
\]
\end{theorem}

\begin{proof}
For $n=1$, the assertion is clear, thus we assume $n \geq 2$.
By \cref{V-module-hom}, we obtain the exact sequence
\[
0 \to \phi_*W_{n-1} \xrightarrow{V} W_n(A) \xrightarrow{R^{n-1}} A \to 0
\]
of $A$-modules.
Since $s_\phi \colon A \to W_n(A)$ is a splitting map of $R^{n-1} \colon W_n(A) \to A$, we obtain an $A$-module isomorphism
\[
W_n(A) \xrightarrow{\sim} \phi_*W_{n-1}(A) \oplus A\ ;\ \alpha \mapsto (\Delta_{W_n}(\alpha),R^{n-1}(\alpha)).
\]
Repeating to take decompositions, we obtain the isomorphism
\[
\Psi_n \colon W_n(A) \xrightarrow{\sim} A \oplus \phi_*A \oplus \cdots \oplus \phi^{n-1}_*A,
\]
and $\Psi_n$ satisfies
\[
\Psi_n((a_0,\ldots,a_{n-1}))=(a_0,\Delta_1(a_0)+a_1,\ldots,\Delta_{n-1}(a_0)+\Delta_{n-2}(a_1)+\cdots+\Delta_1(a_{n-2})+a_{n-1}),
\]
as desired.
\end{proof}

\begin{proposition}\label{s_phi-circ-phi}
We use the notation introduced in \cref{notation:F-lift}.
The $A$-module homomorphism $s_\phi \circ \phi \colon A \to \phi_*W_n(A)$ induces $\overline{s_\phi \circ \phi} \colon A/pA \to \phi_*W_n(A)/pW_n(A)$.
Then $\overline{s_\phi \circ \phi}(\overline{a})=[a^p]$ for $a \in A$ and $\overline{a}$ is the image of $a$ by $A \to A/pA$.
\end{proposition}

\begin{proof}
It is enough to show $\Psi_n(s_\phi \circ \phi(a)) \equiv \Psi_n([a^p]) \mod p$ for $a \in A$.
By \cref{prop:delta map'} (1), we have $\Delta_{W_n}(s_\phi \circ \phi(a))=0$, thus $\Psi_n(s_\phi \circ \phi(a))=(\phi(a),0,\ldots,0)$.
On the other hand, by \cref{prop:delta map'} (3), we have
\[
\Delta_{W_n}([a^p]) \equiv ps_\phi \circ \phi(a^{p-1})\Delta_{W_n}([a]) \equiv 0 \mod pW_{n-1}(A).
\]
Thus, we have $\Psi_n([a^p])=(a^p,0,\ldots,0)$.
Since $\phi(a) \equiv a^p \mod p$, we obtain $\Psi_n(s_\phi \circ \phi(a)) \equiv \Psi_n([a^p]) \mod p$, as desired.
\end{proof}

\section{Quasi-$F$-splitting in mixed characteristic}

\subsection{Definition of quasi-$F$-splitting}
In this subsection, we introduce the notion of quasi-$F$-splitting for rings in mixed characteristic and study first properties of quasi-$F$-splitting.

\begin{proposition}\label{Phi-well-def}
Let $R$ be a $\Z_{(p)}$-algebra and $n$ a positive integer.
We set $\overline{R}:=R/pR$ and $\overline{W}_n(R):=W_n(R)/pW_n(R)$.
Then the map 
\[
\Phi_{R,n} \colon \var{R} \to \var{W}_n(R)\ ;\ a \mapsto [\tilde{a}^p]
\]
is well-defined and a ring homomorphism, where $\tilde{a} \in R$ is a lift of $a \in \var{R}$.
\end{proposition}

\begin{proof}
Let $A:=\Z_{(p)}[X_f \mid f \in R]$ and $\phi \colon A \to A$ be the $\Z_{(p)}$-algebra satisfying $\phi(X_f)=X^p_f$ for every $f \in R$.
Let $\pi \colon A \to R$ be the $\Z_{(p)}$-algebra homomorphism defined by $\pi(X_f)=f$.
By \cref{s_phi-circ-phi}, the map 
\[
\Phi_{A,n} \colon \var{A}:=A/pA \to \var{W}_n(A):=W_n(A)/pW_n(A)\ ;\ a \mapsto [\tilde{a}^p]
\]
is well-defined and a ring homomorphism, where $\tilde{a} \in A$ is a lift of $a \in A$.
We take $a,b \in A$ such that $\pi(a)=0$.
Then we have $[(a+pb)^p] \equiv [a^p] \mod p$ and $[a^p]$ is contained in the kernel of $W_n(\pi)$.
Thus, $\Phi_{A,n}$ induces the ring homomorphism $\Phi_{R,n}$, as desired.
\end{proof}

\begin{notation}\label{notation:p-free-Z_p-alg}
Let $R$ be a $\Z_{(p)}$-algebra.
We set $\overline{R}:=R/pR$ and 
\[
\overline{W}_n(R):=W_n(R)/pW_n(R)
\]
for each positive integer $n$.
We defined an $\var{R}$-module $Q_{R,n}$ by $Q_{R,n}:=\var{W}_n(R)$ as group and the $\var{R}$-module structure is endowed by the ring homomorphism $\Phi_{R,n}$ for each positive integer $n$.
\end{notation}

\begin{proposition}\label{exact-Q}
We use the notation introduced in \cref{notation:p-free-Z_p-alg}.
We assume $R$ is $p$-torsion free.
For positive integers $n,m$, we obtain the following exact sequence as $\var{R}$-modules:
\[
0 \to F^m_*Q_{R,n} \xrightarrow{V^m} Q_{R,n+m} \xrightarrow{R^{n}} Q_{R,m} \to 0.
\]
In particular, if $R$ is Noether and $\var{R}$ is $F$-finite, then $Q_{R,n}$ is a finite $\var{R}$-module for every positive integer $n$.
\end{proposition}

\begin{proof}
First, we prove that $V^m \colon F^m_*Q_{R,n} \to Q_{R,n+m}$ is an $\var{R}$-module homomorphism.
We take $\alpha \in W_{n}(R)$ and $b \in R$.
Then we have $[b^p]V^m\alpha=V([b^{p^{m+1}}]\alpha)$ by \cref{rule-witt}, thus $V^m$ is an $\var{R}$-module homomorphism.
Next, we prove the exactness.
It is enough to show that the sequence
\[
0 \to \var{W}_{n}(R) \xrightarrow{V^m} \var{W}_{n+m}(R) \xrightarrow{R^n} \var{W}_m(R) \to 0
\]
is exact as groups.
It follows from the facts that $W_s(R)$ is $p$-torsion free for each positive integer $s$ and we have the exact sequence
\[
0 \to W_n(R) \xrightarrow{V^m} W_{n+m}(R) \xrightarrow{R^{n}} W_m(R) \to 0
\]
as groups.
The last assertion follows from the exact sequence and $Q_{R,1}=F_*\var{R}$.
\end{proof}

\begin{definition}\label{def:Q}\textup{(cf.~\cite{Yobuko19})}
We use the notation introduced in \cref{notation:p-free-Z_p-alg}.
We assume $R$ is Noether and $\var{R}$ is $F$-finite.
\begin{itemize}
    \item Let $n$ be a positive integer.
    We say that $R$ is \emph{$n$-quasi-$F$-split} if $\Phi_{R,n} \colon \var{R} \to Q_{R,n}$ splits as $\var{R}$-modules.
    Furthermore, we say that $R$ is \emph{quasi-$F$-split} if $R$ is $n$-quasi-$F$-split for some positive integer $n$.
    \item The \emph{quasi-$F$-splitting height} $\mathrm{ht}(R)$ is defined by 
    \[
    \mathrm{ht}(R):=\mathrm{inf}\{n \geq 1 \mid \text{$R$ is $n$-quasi-$F$-split}\}
    \]
    if $R$ is quasi-$F$-split, and $\mathrm{ht}(R):=\infty$ if $R$ is not quasi-$F$-split.
        \item Let $n$ be a positive integer and $f \in R$ a non-zero divisor.
    We define an $\var{R}$-module $Q^f_{R,n}$ by 
    \[
    Q^f_{R,n}:=\Ker(Q_{R,n} \to Q_{R/f,n}).
    \]
    We note that 
    \[
    Q^f_{R,n}=\{(fa_0,\ldots,fa_{n-1}) \in \var{W}_n(R)\mid a_0,\ldots a_{n-1} \in R\}
    \]
    as sets.
    We say that $(R,f)$ is \emph{purely $n$-quasi-$F$-split} if the map
    \[
    \Phi^f_{R,n} \colon \var{R} \to Q^f_{R,n}\ ;\ a \mapsto [af]^p 
    \]
    splits as $\var{R}$-modules.
    \item Let $n,e$ be positive integers.
    We define $\var{R}$-module $Q_{R,n,e}$ by the following pushout diagram as $\var{R}$-modules:
    \[
    \begin{tikzcd}
        F^n_*\var{R} \arrow[r,"F^{e-1}"] \arrow[d,"V^{n-1}"] & F^{n+e-1}_*\var{R} \arrow[d] \\
        Q_{R,n} \arrow[r] & Q_{R,n,e} \arrow[ul, phantom, "\ulcorner", very near start]
    \end{tikzcd}
    \]
    and the induced $\var{R}$-module homomorphism is denoted by
    \[
    \Phi_{R,n,e} \colon R \xrightarrow{\Phi_{R,n}} Q_{R,n} \to Q_{R,n,e}.
    \]
    \item We say that $R$ is \emph{quasi-$(F,F^{\infty})$-split} if $R$ is quasi-$F$-split and $\Phi_{R,\mathrm{ht}(R),e}$ splits as $\var{R}$-modules for every positive integer $e$. 
\end{itemize}
\end{definition}

\begin{remark}
If $R$ is an $\F_p$-algebra, then the notion of quasi-$F$-splitting coincides with that in \cite{Yobuko19}.
\end{remark}

\begin{proposition}\label{compsre-posi-case}
We use the notation introduced in \cref{notation:p-free-Z_p-alg}.
We assume $R$ is Noetherian $p$-torsion free and  $\var{R}$ is $F$-finite.
Then we have $\mathrm{ht}(R) \leq \mathrm{ht}(\var{R})$.
Furthermore, if $R$ is normal and $\sht(R)=\sht(\var{R})=\sht^{e}(\var{R})$ for every $e \geq 1$, where $\sht^e(\var{R})$ is defined in \cite{TWY}*{Definition~3.5}, then $R$ is quasi-$(F,F^{\infty})$-split.
\end{proposition}

\begin{proof}
The first assertion follows from the following diagram:
\[
\begin{tikzcd}[row sep=0.3cm, column sep=2cm]
    \var{R} \arrow[r,midway,"\Phi_{R,n}"] \arrow[rd,near end,"\Phi_{\var{R},n}"'] & Q_{R,n} \arrow[d] \\
     & Q_{\var{R},n}
\end{tikzcd}
\]
for each positive integer $n$.
We set $n:=\sht(R)$.
Since we have the commutative diagram
\[
\begin{tikzcd}
    F^e_*\var{R} \arrow[r,"F^e"] \arrow[d,"V^{n-1}"] & F^{n+e-1}_*\var{R}  \arrow[d,"V^{n-1}"] \\
    Q_{R,n} \arrow[r,"F^e"] & Q^e_{\var{R},n},
\end{tikzcd}
\]
we obtain the $\var{R}$-module homomorphism $Q_{R,n,e} \to Q^e_{\var{R},n}$, and in particular, we have 
\begin{align*}
    \Hom_{W_n(\var{R})}(Q^e_{\var{R},n},W_n\omega_{\var{R}}(-K_{\var{R}})) &\to \Hom_{W_n(\var{R})}(Q_{R,n,e},W_n\omega_{\var{R}}(-K_{\var{R}})) \\
    &\simeq \Hom_{\var{R}}(Q_{R,n,e},\var{R}) \to \Hom_{\var{R}}(\var{R},\var{R}) \simeq \var{R}.
\end{align*}
Since the composition is surjective by $\sht^e(\var{R})=n$, the homomorphism $\var{R} \to Q_{R,n,e}$ splits, as desired.
\end{proof}

\begin{proposition}\label{qfs:comp}
We use the notation introduced in \cref{notation:p-free-Z_p-alg} and assume $(R,\m)$ is Noether local $p$-torsion free and $\var{R}$ is $F$-finite.
Then we have $\sht(R)=\sht(\widehat{R})$, where $\widehat{R}$ is the $m$-adic completion of $R$.
Furthermore, $R$ is quasi-$(F,F^\infty)$-split if and only if so is $\widehat{R}$.
\end{proposition}

\begin{proof}
First, we obtain $F^e_*Q_{R,n} \otimes_{R} \widehat{R} \simeq F^e_*Q_{\widehat{R},n}$ for positive integers $e,n$ by \cref{exact-Q} and an inductive argument.
Therefore, we obtain the first assertion.
Furthermore, since taking pushout does not change the cokernels, we obtain the exact sequence
\[
0 \to Q_{R,n} \to Q_{R,n,e} \to F^e_*B^{e-1}_{\var{R}} \to 0,
\]
thus we obtain $Q_{R,n,e} \otimes_{R} \widehat{R} \simeq Q_{\widehat{R},e,n}$, and in particular, we have the last assertion.
\end{proof}

\begin{proposition}\label{CM-Q}
We use the notation introduced in \cref{notation:p-free-Z_p-alg} and assume $R$ is $p$-torsion free.
If $R$ is Cohen-Macaulay, then so is $Q_{R,n}$ for every positive integer $n$.
\end{proposition}

\begin{proof}
We prove the assertion by induction on $n$.
Since $R$ is Cohen-Macaulay and $p$-torsion free, the ring $\var{R} \simeq W_{R,1}$ is Cohen-Macaulay.
We take an integer $n \geq 2$.
By \cref{exact-Q}, we obtain an exact sequence 
\[
0 \to F_*Q_{R,n-1} \to Q_{R,n} \to \var{R} \to 0
\]
as $\var{R}$-modules.
By the induction hypothesis, the modules $F_*Q_{R,n-1}$ and $\var{R}$ are Cohen-Macaulay, thus so is $Q_{R,n}$, as desired.
\end{proof}

\begin{theorem}\label{thm:adj-inv-qfs}
We use the notation introduced in \cref{notation:p-free-Z_p-alg} and assume $(R,\m)$ is local.
Let $f \in \m$ be a regular element such that $R/f$ is $p$-torsion free.
We assume $R$ is Gorenstein.
Then $(R,f)$ is purely quasi-$F$-split if and only if $R/f$ is quasi-$F$-split. 
\end{theorem}

\begin{proof}
We obtain the following diagram in which each horizontal sequence is exact:
\[
\begin{tikzcd}
    0 \arrow[r] & \var{R} \arrow[r,"\cdot f"] \arrow[d,"\Phi^f_{R,n}"] & \var{R} \arrow[r] \arrow[d,"\Phi_{R,n}"] & \var{R/f} \arrow[r] \arrow[d,"\Phi_{R/f,n}"] & 0 \\
    0 \arrow[r] & Q^f_{R,n} \arrow[r] & Q_{R,n} \arrow[r] & Q_{R/f,n} \arrow[r] & 0,
\end{tikzcd}
\]
where $\var{R/f}:=R/(f,p)$.
Taking local cohomology we obtain the following diagram in which each horizontal sequence is exact:
\[
\begin{tikzcd}
    0 \arrow[r] & H^{d-1}_\m(\var{R/f}) \arrow[r] \arrow[d,"\Phi_{R/f,n}"] & H^d_\m(\var{R}) \arrow[r,"\cdot f"] \arrow[d,"\Phi^f_{R,n}"] & H^d_\m(\var{R}) \arrow[r] \arrow[d,"\Phi_{R,n}"] & 0 \\
    0 \arrow[r] & H^{d-1}_\m(Q_{R/f,n}) \arrow[r,"(\star)"] & H^d_\m(Q^f_{R,n}) \arrow[r] & H^d_\m(Q_{R,n}) \arrow[r] & 0,
\end{tikzcd}
\]
where $d:=\dim \var{R}$ and the injectivity of $(\star)$ follows from \cref{CM-Q}.
Since the socle of $H^d_\m(\var{R})$ is contained in the image of the map $H^{d-1}_\m(\var{R/f}) \to H^d_\m(\var{R})$, we obtain that $\Phi_{R/f,n}$ is injective if and only if so is $\Phi^f_{R,n}$, as desired. 
\end{proof}

\subsection{Quasi-$F$-splitting for perfectoids}
In this subsection, we study the purity of $\Phi_{R,n}$ for a perfectoid $R$.
\cref{thm:purity-perfd} is related to the computation of perfectoid pure threshold (cf.~\cref{equiv-qfs-p-pure}).

\begin{theorem}\label{thm:purity-perfd}\textup{(cf.~\cite{BMS18}*{Lemma~3.12}}
Let $R$ be a $p$-torsion free perfectoid and $n$ a positive integer.
We assume $R$ has a compatible system of $p$-power roots $\{p^{1/p^e}\}$ of $p$.
Then there exist an $\var{R}$-module homomorphism $\var{\psi}_n \colon Q_{R,n} \to R/(p^{i_n})$ and a unit $w \in \var{R}$ such that $\var{\psi}_n(1)=1$ and $\var{\psi}_n(V^{n-1}[a])=wp^{i_{n-1}}F^{-n}(a)$ for every $a \in R$, where 
\[
i_m:=\frac{1}{p}+\cdots+\frac{1}{p^m}
\]
for every integer $m \geq 1$ and $i_0:=0$.
\end{theorem}

\begin{proof}
Let $R^\flat$ be a tilt of $R$ and $T \in R^\flat$ the element corresponding to a system $\{p^{1/p^e}\}$.
Let $A:=W(R^\flat)$ and $d:=[T]-p$, then $(A,(d))$ is a perfect prism with $A/(d) \simeq R$ by \cite{BS}*{Corollary~2.31}.
The lift of Frobenius on $A$ is denoted by $\phi$.
We prove the following claim:
\begin{claim}\label{cl:to-unit}
There exists an $A$-module homomorphism $\psi_n \colon Q_{A,n} \to R^\flat$ such that
\begin{enumerate}
    \item[(a)] $\psi_n(1)$ and $\psi_n(\Delta_{W_{n+1}}(d))$ are units, 
    \item[(b)] $\psi_n(W_n((d))\cdot \var{W}_n(A)) \subseteq (T^{i_n})$, and
    \item[(c)] there exists a unit $w \in R^\flat$, $\psi_n(V^{n-1}[a])=wT^{i_{n-1}}F^{-n}(a)$ for every $a \in A$.
\end{enumerate}
\end{claim}
\begin{claimproof}
We prove the assertions by induction on $n$.
For $n=1$, we define $\psi_1$ by
\[
\psi_1 \colon Q_{A,1} \simeq F_*R^\flat \xrightarrow{F^{-1}} R^\flat, 
\]
then $\psi_1(1)=1$ and 
\[
\psi_1(\Delta_{W_2}(d))=\psi_1(\delta_1(d))=F^{-1}(\delta_1(d)).
\]
Since $\delta_1(d)$ is a unit, the map $\psi_1$ satisfies  condition (a).
On the other hand, we have 
\[
\psi_1(W_1((d))\cdot \var{W}_1(A))=F^{-1}(TR^\flat)=T^{1/p}R^\flat,
\]
thus $\psi_1$ also satisfies  condition (b).
Furthermore, we have $\psi(a)=F^{-1}(a)$, thus $\psi$ satisfies condition (c).
Next, we assume $n \geq 2$.
By the induction hypothesis, there exists $\psi_{n-1} \colon Q_{A,n-1} \to R^\flat$ such that $\psi_{n-1}$ satisfies conditions (a), (b), (c).
We replace $\psi_{n-1}$ by $(-\psi_{n-1}(\Delta_{W_{n}}(d))^{-1}) \cdot \psi_{n-1}$, we may assume that $\psi_{n-1}(\Delta_{W_{n}}(d))=-1$.
We set
\[
\psi':=T^{1/p} \cdot F^{-1} \circ \psi_{n-1} \colon F_*Q_{A,n-1} \to R^\flat.
\]
By the proof of \cref{decomposition-W}, we obtain an isomorphism
\[
Q_{A,n} \xrightarrow{\sim} F_*Q_{A,n-1} \oplus F_*R^\flat\ ;\ \alpha \mapsto (\Delta(\alpha),R^{n-1}(\alpha)),
\]
and we define an $A$-module homomorphism $\psi_n$ by
\[
\psi_n \colon \phi_*\var{W}_n(A) \to \phi^2_*W_{n-1}(A) \oplus F_*R^\flat \xrightarrow{(\psi',F^{-1})} R^\flat,
\]
that is, for $\alpha=(a_0,\ldots,a_{n-1}) \in \phi_*\var{W}_n(A)$, 
\[
\psi_n(\alpha)=\psi'(\Delta_{W_{n}}(\alpha))+F^{-1}(a_0).
\]
We take $a_0,\ldots,a_{n-1} \in A$, we have
\begin{align*}
    &\psi_n((a_0d,\ldots,a_{n-1}d)) \\
    =& \psi'(\Delta_{W_{n}}([a_0d])+(a_1d,\ldots,a_{n-1}d))+F^{-1}(a_0d) \\
    =& F^{-1}(a_0)(\psi'(\Delta_{W_{n}}(d))+T^{1/p}) +\psi'([d^p]\Delta_{W_n}([a_0]) +(a_1d,\ldots,a_{n-1}d)) \\
    \overset{(\star_1)}{=}&T^{1/p}\psi'(\Delta_{W_n}([a_0]))+\psi'((a_1d,\ldots,a_{n-1}d)) \overset{(\star_2)}{\in} (T^{i_n}),
\end{align*}
where $(\star_1)$ follows from
\[
\psi'(\Delta_{W_n}(d))+T^{1/p})=T^{1/p}(F^{-1}\circ \psi_{n-1}(\Delta_{W_n}(d))+1)=0
\]
and $(\star_2)$ follows from the facts
\[
T^{1/p}\psi'(\Delta_{W_n}([a_0])) \in (T^{2/p}) \subseteq (T^{i_n})
\]
and 
\begin{align*}
    \psi'((a_1d,\ldots,a_{n-1}d)) &=T^{1/p}F^{-1} \circ \psi_{n-1}((a_1d,\ldots,a_{n-1}d)) \\
    &\subseteq T^{1/p} \cdot F^{-1}((T^{i_{n-1}})) =(T^{i_n}).
\end{align*}
Thus, the homomorphism $\psi_n$ satisfies condition (b).
Furthermore, we have $\psi_n(1)=F^{-1}(1)=1$ and 
\begin{align*}
    \psi_n(\Delta_{W_{n+1}}(d)) &= \psi'(\Delta_{W_n} \circ \Delta_{W_{n+1}}(d))+F^{-1}(\delta_1(d)) \\
    &\equiv F^{-1}(\delta_1(d)) \mod T^{1/p}.
\end{align*}
Since $R^\flat$ is $T^{1/p}$-complete and $F^{-1}(\delta_1(d))$ is a unit, the element $\psi_n(\Delta_{W_{n+1}}(d))$ is a unit.
Furthermore, we have
\[
\psi(V^{n-1}[a])=\psi'(V^{n-2}[a])=T^{1/p}\cdot F^{-1}(wT^{i_{n-2}}F^{-(n-1)}(a))=F^{-1}(w)T^{i_{n-1}}F^{-n}(a),
\]
thus, the map $\psi_n$ satisfies  conditions (a), (b), (c).
\end{claimproof} \\
We take $\psi_n$ as in \cref{cl:to-unit}.
Replacing $\psi_n$ by $(\psi_n(1))^{-1} \cdot \psi_n$, we may assume $\psi_n(1)=1$.
By  condition (b), $\psi_n$ induces the map 
\[
\var{\psi}_n \colon Q_{R,n} \to R^\flat/(T^{i_n}) \simeq R/(p^{i_n})
\]
with $\var{\psi}_n(1)=1$.
Furthermore, by  condition (c), there exists a unit $w \in \var{R}$ such that $\var{\psi}(V^{n-1}[a])=wT^{i_n}F^{-n}(a)$ for every $a \in A$.
\end{proof}

\begin{remark}
By \cite{BMS18}*{Lemma 3.12}, the homomorphism $\theta_n$ induces the isomorphism
\[
\var{\theta_n} \colon R^\flat/((p^\flat)^{1+1/p+\cdots+1/p^{n-1}}) \xrightarrow{\sim} \var{W}_n(R).
\]
Consider the composition
\[
Q_{R,n} \xrightarrow{\var{\theta_n}^{-1}} F_*R^\flat/((p^\flat)^{1+1/p+\cdots+1/p^{n-1}}) \xrightarrow{F^{-1}} R^\flat/((p^\flat)^{(1/p+\cdots+1/p^{n})}) \simeq R/(p^{i_n}),
\]
then we get a splitting of $\Phi_{R,n}$.
\end{remark}

\subsection{Fedder-type criterion for quasi-$F$-splitting}
In this subsection, we introduce Fedder-type criterion (\cref{Fedder}) for quasi-$F$-splitting and quasi-$(F,F^\infty)$-splitting.
The proof of \cref{Fedder} is almost identical to a proof of \cite{TWY}*{Theorem~A}, however, for the convenience of the reader, we provide a proof in Appendix \ref{proof-fedder}.
Furthermore, we introduce examples of computations of quasi-$F$-splitting height and quasi-$(F,F^\infty)$-splitting by using \cref{Fedder}.

\begin{theorem}\label{Fedder}\textup{(cf.~\cite{kty}*{Theorem~A})}
Let $(A,\m)$ be a regular local ring with lift of Frobenius $\phi \colon A \to A$ such that $\phi$ is finite.
We assume $p \in \m \backslash \m^2$ and set $\var{A}:=A/pA$, then $\var{A}$ is regular.
We fix a generator $u \in \Hom_{\var{A}}(F_*\var{A},\var{A})$ as $F_*\var{A}$-module.
Let $f_1,\ldots,f_r$ be a regular sequence in $A$ and we set $I:=(f_1,\ldots,f_r)\var{A}$, $f:=f_1 \cdots f_r$.
We assume $R:=A/(f_1,\ldots,f_r)$ is $p$-torsion free.
\begin{enumerate}
    \item We define a sequence $\{I_n\}$ of ideals of $\var{A}$ by $I_1:=f^{p-1}\var{A}+I^{[p]}$ and
    \[
    I_n:=u\left(F_*(\Delta_1(f^{p-1})I_{n-1})\right)+I_1
    \]
     for $n \geq 2$, inductively.
    Then we have
    \[
    \sht(R)=\inf\{n \geq 1 \mid I_n \nsubseteq \m^{[p]}\},
    \]
    where if $I_n \subseteq \m^{[p]}$ for every positive integer $n$, then the right hand side is defined by $\infty$.
    \item We set $I':=\bigcap_{e \geq 1} u^e\left(F^e_*(f^{p^e-1}\var{A})\right)$.
    We define a sequence $\{I'_n\}$ of ideals of $\var{A}$ by $I'_1:=f^{p-1}I'+I^{[p]}$ and
    \[
    I'_n:=u\left(F_*(\Delta_1(f^{p-1})I'_{n-1})\right)+f^{p-1}\var{A}+I^{[p]}
    \]
     for $n \geq 2$, inductively.
     Then $R$ is quasi-$(F,F^{\infty})$-split if and only if $R$ is quasi-$F$-split and $I'_{\sht(R)} \nsubseteq \m^{[p]}$.
\end{enumerate}    
\end{theorem}

\begin{example}\label{ex:fedder}\textup{(cf.~\cref{ex:threshold})}
Let $A:=\mathbb{Z}_{(p)}[[X_1,\ldots,X_N]]$ and the $\Z_{(p)}$-algebra homomorphism $\phi$ is given by $\phi(X_i)=X_i^p$.
If $N=3$, we set $x:=X_1,y:=X_2,z:=X_3$, and if $N=4$ we set $x:=X_1,y:=X_2,z:=X_3,w:=X_4$.
Then $F_*\var{A}$ has a basis $\{F_*(X_1^{i_1} \cdots X_{N}^{i_N}) \mid p-1 \geq i_1,\ldots,i_N \geq 0 \}$.
The corresponding dual basis is denoted by $u_1,\ldots,u_d$ and the corresponding element to $X_1^{p-1} \cdots X_N^{p-1}$ is denoted by $u$, then $u$ is a generator of $\Hom_{\var{A}}(F_*\var{A},\var{A})$ as $F_*\var{A}$-module.
Let $f \in A$ be a nonzero divisor such that $R:=A/f$ is $p$-torsion free.
We use same notations $I_n, I'_n,I'$ as in \cref{Fedder} for every positive integer $n$.
We set $\var{R}:=R/pR$ and
 define $\theta \colon F_*\var{A} \to \var{A}$ by $\theta(F_*a)=u(F_*(a\Delta_1(f)))$.
\begin{enumerate}
    \item Let $p=2, N=3, f=x^3+y^3+z^3$.
    Then $I_1 \subseteq \m^{[p]}$ and $u\left(F_*(f\Delta_1(f)\right) \equiv xyz \mod \m^{[p]}$, thus we have $I_2 \nsubseteq \m^{[p]}$.
    Furthermore, we have $u\left(F_*(f\Delta_1(f)\m)\right) \subseteq \m^{[p]}$, thus we have $I_2' \subseteq \m^{[p]}$, in conclusion, $R$ is not quasi-$(F,F^{\infty})$-split (cf.~\cref{CY-case}).
    \item Let $N=3, f=z^2+x^3+y^5$.
    \begin{itemize}
        \item We consider the case of $p=2$, then we have
    \begin{equation}\label{eq:rdp1}
    \begin{tikzcd}
        zf \arrow[r,"\theta",mapsto] & xy^2z \arrow[r,"\theta",mapsto] & y^3z \arrow[r,"\theta",mapsto] & xyz \notin \m^{[p]},
    \end{tikzcd}
    \end{equation}
    thus we have $\sht(R) \leq 4$.
    Furthermore, since we have $f\Delta_1(f)^3 \in \m^{[p^3]}$, thus we have $I_3 \subseteq \m^{[p]}$, therefore $\sht(R)=4$.
    Moreover, since $I'=(x,y^2,z)$ and the first element $zf$ of (\ref{eq:rdp1}) is contained in $I_1'$, the ring $R$ is quasi-$(F,F^{\infty})$-split.
    \item We consider the case of $p=3$, then we have
    \begin{equation*}
    \begin{tikzcd}
        x^2yf^2 \arrow[r,"\theta",mapsto] & -x^2y^3f \arrow[r,"\theta",mapsto] & x^2y^2f \equiv x^2y^2z^2 \notin \m^{[p]},
    \end{tikzcd}
    \end{equation*}
    $f^2\Delta_1(f^2) \in \m^{[p^2]}$, and $I'=(x,y^3,z)$, thus we have $\sht(R)=3$ and $R$ is quasi-$(F,F^\infty)$-split.
    \item We consider the case of $p=5$, then we have
    \begin{equation*}
    \begin{tikzcd}
        y^4f^4 \arrow[r,"\theta",mapsto] & xy^4f^3 \equiv -2x^4y^4z^4 \notin \m^{[p]},
    \end{tikzcd}
    \end{equation*}
    $f^4 \in \m^{[p]}$, and $I'=(x,y,z)$, thus we have $\sht(R)=2$ and $R$ is quasi-$(F,F^\infty)$-split.
    \end{itemize}
    \item Let $p=2, N=4, f=w^2+xyz(x+y+z)$.
    Then $I_1 \subseteq \m^{[p]}$ and 
    \[
    u(F_*(zwf\Delta_1(f)))=xyzw \notin \m^{[p]},
    \]
    thus we have $\sht(R)=2$.
    We note that $\sht(\var{R})=\infty$ by \cite{kty}*{Example~7.11}.
    Furthermore, the ring $R$ is quasi-$(F,F^{\infty})$-split by $zw \in I'$.
    \item Let $p=2, N=4, f=w^2+xyz(x+y+z)+p(xy+xz+yz)w$.
    Then $\sht(R)=3$.
    Indeed, 
    \begin{eqnarray*}
        I_2&=&(f,zw^2+x^2yz^2+xy^2z^2+xz^2w+yz^2w,xw^2+x^2y^2z+x^2yz^2+x^2yw+x^2zw, \\
        &&xyw^2+xzw^2+yzw^2+w^3+x^2y^2z^2) \subseteq \m^{[p]},
    \end{eqnarray*}
    and
    \[
    \theta(F_*(xw^2+x^2y^2z+x^2yz^2+x^2yw+x^2zw)) =xzw+xyw \notin \m^{[p]},
    \]
    thus we have $\sht(R)=3$.
    On the other hand, we have
    \begin{equation*}
    \begin{tikzcd}
        xwf\Delta_1(f) \arrow[r,"\theta",mapsto] & xw^2+x^2y^2z+x^2yz^2+x^2yw+x^2zw  \arrow[r,"\theta",mapsto] & xzw+xyw 
    \end{tikzcd}
    \end{equation*}
    and $xw \in I'$,
    thus  $R$ is quasi-$(F,F^\infty)$-split.
    \item Let $p=2, N=3, f=z^2+x^2y+xy^n$ for $n \geq 2$.
    First, we note that $z \in I'$.
    By the argument in \cite{kty}*{Section~7.1}, we have $\theta^m\left(F^m_*(zy^{2^m-n}f)\right)=xyz \notin \m^{[p]}$ and $\sht(R)=m+1$, where $m:=\rup{\log_2 n}$, thus $R$ is quasi-$(F,F^\infty)$-split.
\end{enumerate}
\end{example}

\begin{example}\label{ex:non-qfs}
Let $A:=\Z_{(p)}[[x,y,z,x',y',z']]$, $f:=x^3+y^3+z^3$, $f':=x'^3+y'^3+z'^3$, $R:=A/(f,f')$, and $p \equiv 2 \mod 3$.
Then $R$ is not quasi-$F$-split.
Indeed, we have
\[
\Delta_1\left((ff')^{p-1}\right)=f^{p(p-1)}\Delta_1(f'^{p-1})+f'^{p(p-1)}\Delta_1(f^{p-1}) \in \m^{[p^2]},
\]
thus we have $I_n \subseteq \m^{[p]}$ for every integer $n \geq 1$.
\end{example}

\section{Quasi-$F$-splitting versus perfectoid purity and  perfectoid pure threshold}
In this section, we study the relationship between quasi-$F$-splitting and perfectoid purity.
For this purpose, we introduce a procedure (\cref{F-W-program}) to obtain a sequence of elements of local cohomologies.

\begin{notation}\label{notation:compare}
Let $(R,\m)$ be a Noetherian local $\Z_{(p)}$-algebra such that $R$ is $p$-torsion free, $p \in \m$, and $\var{R}:=R/pR$ is $F$-finite.
We set $d:=\dim{\var{R}}$.
\end{notation}

\subsection{Quasi-$F$-splitting versus perfectoid purity}

\begin{proposition}\label{p-pure-to-qfs}
We use the notation introduced in \cref{notation:compare}.
Let $n$ be a positive integer and $i_n=1/p+\cdots+1/p^{n}$.
We assume $R$ is Gorenstein.
Let $\rho \colon R \to R_{\infty}$ be a ring homomorphism to a $p$-torsion free perfectoid $R_\infty$.
If $\sht(R) \geq n+1$, then 
\[
H^{d+1}_\m(R) \to H^{d+1}_\m(R_\infty) \xrightarrow{\cdot p^{1-i_n}} H^{d+1}_\m(R_\infty)
\]
is not injective.
In particular, we have $\ppt(R;\div(p)) \leq 1-i_n$.
\end{proposition}

\begin{proof}
We consider the following commutative diagram
\[
\begin{tikzcd}
    H^d_\m(\var{R}) \arrow[r] \arrow[d,"H^d_\m(\rho)"] & H^d_\m(Q_{R,n})  \arrow[d] & \\
    H^d_\m(\var{R}_\infty)  \arrow[r] & H^d_\m(Q_{R_\infty,n})  \arrow[r,"H^d_\m(\var{\psi}_n)"] & H^d_\m(\var{R}_\infty/(p^{i_n})), 
\end{tikzcd}
\]
where $\var{\psi}_n$ is as in \cref{thm:purity-perfd}.
In particular, taking a non-zero element $\eta \in H^d_\m(\var{R})$ in the socle of $H^d_\m(\var{R})$, since $\eta$ is zero in $H^d_\m(Q_{R,n})$ by $\sht(R) \geq n+1$, we have 
\[
H^d_\m(\var{R}) \to H^d_\m(\var{R}_\infty)  \rightarrow H^d_\m(\var{R}_\infty/(p^{i_n})) \ ;\ \eta \mapsto \rho(\eta) \mapsto 0.
\]
By the commutative diagram in which each horizontal sequence is exact;
\[
\begin{tikzcd}
    0 \arrow[r] & R_\infty \arrow[r,"\cdot p"] \arrow[d,"\cdot p^{1-i_n}"] & R_{\infty} \arrow[r] \arrow[d,equal] & \var{R}_\infty \arrow[r] \arrow[d] & 0 \\
    0 \arrow[r] & R_\infty \arrow[r,"\cdot p^{i_n}"] & R_\infty \arrow[r] & \var{R}_\infty/(p^{i_n}) \arrow[r] & 0,
\end{tikzcd}
\]
we obtain the commutative diagram
\[
\begin{tikzcd}
    H^d_\m(\var{R}_\infty)  \arrow[r] \arrow[d] & H^{d+1}_\m(R_\infty)   \arrow[d,"\cdot p^{1-i_n}"] \\
    H^d_\m(\var{R}_\infty/(p^{i_n}))  \arrow[r] & H^{d+1}_\m(R_\infty).
\end{tikzcd}
\]
Therefore, by the homomorphism
\begin{equation}\label{eq:map-gor}
    H^d_\m(\var{R}) \to H^{d+1}_\m(R) \xrightarrow{H^{d+1}_\m(\rho)} H^{d+1}_\m(R_\infty) \xrightarrow{\cdot p^{1-i_n}} H^{d+1}_\m(R_\infty) 
\end{equation}
the element $\eta$ maps to $0$.
Since $R$ is $p$-torsion free and Gorenstein, the first homomorphism in (\ref{eq:map-gor}) is injective, thus we obtain that the map
\[
H^{d+1}_\m(R) \to H^{d+1}_\m(R_\infty) \xrightarrow{\cdot p^{1-i_n}} H^{d+1}_\m(R_\infty)
\]
is not injective, as desired.
The last assertion follows from the existence of $p$-torsion free test perfectoids (\cref{thm:const-test-perfd}) and \cref{test-perf-tests}.
\end{proof}

\begin{corollary}\label{non-qfs-to-non-bcm}
We use the notation introduced in \cref{notation:compare}.
We assume $R$ is Gorenstein.
If $R$ is not quasi-$F$-split and $p=2$, then $R$ is not perfectoid BCM-regular.
\end{corollary}

\begin{proof}
If $R$ is perfectoid BCM-regular, then there exists $\varepsilon \in \Z[\frac{1}{p}]_{>0}$ such that $(R,\varepsilon \div(p))$ is perfectoid BCM-regular by \cite{bcm-reg}*{Proposition~6.10}.
Thus, there exists a perfectoid big Cohen-Macaulay algebra $B$ over $R$ such that $R \to B \xrightarrow{\cdot p^\varepsilon} B$ is pure.
Since $B$ is $p$-torsion free, it contradicts the assumption by \cref{p-pure-to-qfs}.
\end{proof}

\begin{definition}\label{F-W-program}
We use the notation introduced in \cref{notation:compare} and assume that $R$ is quasi-$F$-split.
We take $\eta \in H^d_\m(\var{R}) \backslash \{0\}$ and construct sequences $\{\eta_i \in H^d_\m(\var{R}) \backslash \{0\} \}_{i \geq 0}$ and $\{n_i \in \Z_{\geq 1}\}_{i \geq 0}$ using the following procedure.
We start with $\eta_0:=\eta$.
Suppose that $\eta_i \in H^d_\m(\var{R})\backslash \{0\}$ has been defined for some $i \geq 0$.
We take the minimal integer $n_i \geq 1$ such that $\Phi_{R,n_i}(\eta_i) \neq 0$.
We note that such an integer $n_i$ exists since $R$ is quasi-$F$-split and $n_i \leq \sht(R)$.
By the exact sequence in \cref{exact-Q}, there exists $\eta' \in H^d_\m(F^n_*\var{R})$ such that $V^{n_i-1}(\eta')=\Phi_{R,n_i}(\eta_i)$.
We choose such an $\eta'$ and set $\eta_{i+1}:=\eta'$.
\end{definition}

\begin{theorem}\label{thm:qfs-to-p-pure}
Let $(A,\m)$ be a local Gorenstein $\Z_{(p)}$-algebra such that $A$ is $p$-torsion free and $p \in \m$.
Let $f \in \m$ be an element of $A$ such that $f,p$ is a regular sequence of $A$.
Let $A_\infty$ be a perfectoid over $A$ such that $A_\infty$ is a cohomologically Cohen-Macaulay, $A \to A_\infty$ is pure, $A_\infty$ has a system of $p$-power roots system of $f$, and $f,p$ are regular elements of $A_\infty$.
Let $R:=A/(f)$ and $R_\infty:=(A_\infty/(f))_{perfd}$.
If $R$ is quasi-$F$-split, then $R \to R_\infty$ is pure.
\end{theorem}

\begin{proof}
We set $\var{A}:=A/(p)$, $\var{R}:=R/(p)$, $\var{R}_\infty:=R_\infty/(p)$, $\var{A}_\infty:=A_\infty/(p)$, and $d:=\dim \var{R}$.
We take a non-zero element $\eta \in H^d_\m(\var{R})$ in the socle of $H^d_\m(\var{R})$.
We construct a sequence $\{\eta_i\}_{i \geq 0}$ with $\eta_0=\eta$ by the procedure in \cref{F-W-program}.
We fix a system of $p$-power roots $\{f^{1/p^e}\}$ of $f$ in $A_\infty$.
By the exact sequence
\[
0 \to \var{A} \xrightarrow{\cdot f} \var{A} \to \var{R} \to 0,
\]
we obtain the homomorphism $H^d_\m(\var{R}) \to H^{d+1}_\m(\var{A})$.
We set
\[
\tau_i :=(u \colon H^d_\m(\var{R}) \to H^{d+1}_\m(\var{A}) \to H^{d+1}_\m(\var{A}_\infty))(\eta_i)
\]
and $\tau_i^{(e)}:=\tau_i f^{1-1/p^e}$ for all integers $e,i  \geq 0$.
\begin{claim}\label{cl:non-van}
We have $\tau_i^{(e)} \neq 0$ for all integers $e,i \geq 0$.
\end{claim}
\begin{claimproof}
We prove the assertion by induction on $e$.
Since $A$ is Cohen-Macaulay, the homomorphism $H^d_\m(\var{R}) \to H^{d+1}_\m(\var{A})$ is injective.
Since $A \to A_\infty$ is pure, the homomorphism $H^{d+1}_\m(\var{A}) \to H^{d+1}_\m(\var{A}_\infty)$ is injective, and in particular, we have $\tau_i^{(0)}=\tau_i \neq 0$ for every integer $i \geq 0$.
Thus, we assume $e \geq 1$ and take $i \geq 0$.
First, we assume $n_i=1$, then we have $F(\eta_i)=\eta_{i+1}$.
By the commutative diagram
\[
\begin{tikzcd}
    H^d_\m(\var{R}) \arrow[r,"u"] \arrow[d,"F"] & H^{d+1}_\m(\var{A}_\infty) \arrow[d,"\cdot f^{p-1} F"] \\
    H^d_\m(F_*\var{R}) \arrow[r,"F_*u"] & H^{d+1}_\m(F_*\var{A}_\infty),
\end{tikzcd}
\]
we have 
\begin{align*}
    \tau_{i+1}^{(e-1)}
    &=u(\eta_{i+1})f^{1-1/p^{e-1}}=u \circ F(\eta_i) f^{1-1/p^{e-1}} \\
    &=F\left(u(\eta_i)\right)f^{p-1} f^{1-1/p^{e-1}}=F\left(u(\eta_i)f^{1-1/p^{e}}\right) \\
    &=F(\tau_i^{(e)}).
\end{align*}
By the induction hypothesis, we have $\tau_{i+1}^{(e-1)} \neq 0$, thus we have $\tau_i^{(e)} \neq 0$, as desired.
Next, we assume $n_i \geq 2$, then $\Phi_{R,n_i}(\eta_i)=V^{n_i-1}(F^{n_i}_*\eta_{i+1})$.
We set $n:=n_i$ and take regular sequence $x_1,\ldots,x_d$ on $\var{R}$ and elements $a,b \in \var{R}$ such that $\eta_i=[a/x_1\cdots x_d]$, $\eta_{i+1}=[b/x_1^{p^n} \cdots x_d^{p^n}]$, and
\[
[a^p] \equiv V^{n-1}(b) \mod ([x_1],\ldots,[x_d])Q_{R,n}.
\]
Therefore, there exist elements $c_0,\ldots,c_{n-1} \in \var{A}$ such that
\[
[a^p] + (c_0f,\ldots,c_{n-1}f) \equiv V^{n-1}(b) \mod ([x_1],\ldots,[x_d])Q_{A,n}.
\]
Since we have $(c_0d,\ldots,c_{n-1}f)[f^{p-p^{1/p^{e+n}}}] \in [f]Q_{A_{\infty},n}$, we obtain
\begin{align*}
    [af^{1-1/p^{e+n-1}}]^p  
    \equiv V^{n-1}(bf^{p^n-1/p^{e-1}}) \mod ([x_1],\ldots,[x_d],[f])Q_{A_\infty,n}.
\end{align*}
In particular, we have
\begin{align*}
    \Phi_{A_\infty,n}(\tau_i^{(e+n-1)})
    &= \left[\frac{[af^{1-1/p^{e+n-1}}]^p}{x_1\cdots x_d f}\right] \\
    &= \left[\frac{V^{n-1}(F^n_*(bf^{p^n-1/p^{e-1}}))}{x_1\cdots x_d f}\right] =V^{n-1}\left(F^n_*\left[\frac{bf^{p^n-1/p^{e-1}}}{x_1^{p^n} \cdots x_d^{p^n}f^{p^n}} \right]\right) \\
    &= V^{n-1}\left(F^n_*\left[\frac{bf^{1-1/p^{e-1}}}{x_1^{p^n} \cdots x_d^{p^n}f} \right]\right) =V^{n-1}(F^n_*\tau_{i+1}^{(e-1)}),
\end{align*}
which is non-zero by the induction hypothesis and $H^{d}_\m(\var{A}_\infty)=0$.
Therefore, we obtain $\tau_i^{(e)} \neq 0$ by $e+n-1 \geq e$, as desired.
\end{claimproof}\\
By \cref{cl:non-van}, we have $\tau f^{1-p^{1/p^e}}=\tau_0^{(e)} \neq 0$ for every integer $e \geq 0$.
Next, we prove that $R \to R_\infty$ is pure.
By the argument of \cite{p-pure}*{Proposition~6.5}, it is enough to show that $fA \to (f^{1/p^\infty})$ is pure.
Since $(f^{1/p^\infty}) \simeq \mathrm{colim}_e f^{1/p^e}A_\infty$, it is enough to show that $fA \to f^{1/p^e}A_\infty$ is pure for every $e \geq 0$.
We fix an integer $e \geq 0$.
Since $f$ is a regular element in $A$ and $A_\infty$, it is enough to show that the homomorphism $A \to A_\infty \xrightarrow{\cdot f^{1-1/p^e}} A_\infty$ is pure.
We consider the following commutative diagram in which horizontal sequence is exact:
\[
\begin{tikzcd}
    0 \arrow[r] & H^{d+1}_\m(\var{A}) \arrow[r,"\alpha"] \arrow[d,"\cdot f^{1-1/p^e}"] & H^{d+2}_\m(A) \arrow[d,"\cdot f^{1-1/p^e}"] \\
    0 \arrow[r] & H^{d+1}_\m(\var{A}_\infty) \arrow[r] & H^{d+2}_\m(A_\infty).
\end{tikzcd}
\]
Since the image $\tau'$ of $\eta$ by the map $H^{d}_\m(\var{R}) \to H^{d+1}_\m(\var{A})$ is an element of the socle of $H^{d+1}_\m(\var{A})$,  we have
\[
(H^{d+1}_\m(\var{A}) \xrightarrow{\cdot f^{1-1/p^e}}H^{d+1}_\m(\var{A}_\infty))(\tau')=\tau f^{1-1/p^e} \neq 0,
\]
and the image of $\tau'$ by $H^{d+1}_\m(\var{A}) \to H^{d+2}_\m(A)$ is contained in the socle of $H^{d+2}_\m(A)$, we obtain the homomorphism $H^{d+2}_\m(A) \xrightarrow{\cdot f^{1-1/p^e}} H^{d+2}_\m(A_\infty)$ is injective.
Since $A$ is Gorenstein, we obtain that the map $A \xrightarrow{\cdot f^{1-1/p^e}} A_\infty$ is pure, as desired.
\end{proof}

\begin{theorem}\label{thm:qfs-to-p-pure-comp}\textup{(cf.~\cref{intro:ht-to-p-pure})}
Let $(R,\m)$ be a local $\Z_{(p)}$-algebra such that $R$ is $p$-torsion free and $p \in \m$.
We assume $R$ is complete intersection.
If $R$ is quasi-$F$-split, then $R$ is perfectoid pure.
In particular, if $R/pR$ is quasi-$F$-split, then $R$ is perfectoid pure.
\end{theorem}

\begin{proof}
By \cref{qfs:comp} and \cite{p-pure}*{Lemma~4.13}, we may assume that $R$ is complete.
Therefore, there exist a regular local ring $R_0$ and regular sequence $f_1,\ldots,f_r$ such that $R_0/(f_1,\ldots,f_r) \simeq R$.
Let $R_i:=R_0/(f_1,\ldots,f_i)$ for $r \geq i \geq 1$.
We take a perfectoid $R_{0,perfd}$ as in \cite{p-pure}*{Section~4.3}.
By Andre's flatness lemma \cite{BS}*{Theorem~7.14}, there exists a $p$-completely faithfully flat extension $R_{0,perfd} \to R_{0,\infty}$ to a perfectoid such that $R_{0,\infty}$ has systems of $p$-power roots of $f_1,\ldots,f_r$.
We set $R_{i,\infty}:=(R_{0,\infty}/(f_1,\ldots,f_i))_{perfd}$.
We prove that $R_i \to R_{i,\infty}$ is pure for each $r \geq i \geq 0$ by induction on $i$.
Since $R_0 \to R_{0,\infty}$ is faithfully flat, we obtain the case of $i=0$.
We take $i \geq 1$ and assume $R_{i-1} \to R_{i-1,\infty}$ is pure.
By \cref{thm:qfs-to-p-pure}, it is enough to show that $R_{i-1,\infty}$ is cohomologically Cohen-Macaulay and $f_i,p$ are regular elements of $R_{i-1,\infty}$.
The cohomologically Cohen-Macaulayness follows from \cite{p-pure}*{Theorem~4.24}.
Furthermore, the sequence $p,f$ is a regular sequence on $R_{i-1,\infty}$ by the proof of \cite{p-pure}*{Theorem~4.24}, and in particular, the element $p$ is a regular element.
By \cref{reg-seq-comm}, the element $f$ is a regular element of $R_{i-1,\infty}$, as desired.
The last assertion follows from the first assertion and \cref{compsre-posi-case}.
\end{proof}

\begin{lemma}\label{reg-seq-comm}
Let $(R,\m)$ be a Noetherian local ring and $M$ an $R$-module.
For elements $f,g \in \m$, if a sequence $f,g$ is a regular sequence on $M$ and $M$ is $(f)$-separated, then $g$ is a regular element.
\end{lemma}

\begin{proof}
We take $a \in M$ such that $ag=0$.
Since $f,g$ is a regular sequence, so is $f^n,g$ for every integer $n \geq 1$.
Since $ag=0$ in $M/f^nM$, we have $a \in f^nM$ for every integer $n \geq 1$.
Since $M$ is $(f)$-separated, we have $a=0$, as desired.
\end{proof}

\subsection{Computation of perfectoid threshold via quasi-$F$-splitting}

\begin{theorem}\label{const-seq}
We use the notation introduced in \cref{notation:compare} and assume that $R$ is quasi-$F$-split and complete intersection.
We take a non-zero element $\eta \in H^d_\m(\var{R})$ in the socle of $H^d_\m(\var{R})$.
We construct a sequence $\{\eta_i\}_{i \geq 0}$ with $\eta_0=\eta$ and $\{n_i\}_{i \geq 0}$ using the procedure in \cref{F-W-program}.
Then we have
\[
\ppt(R;\div(p))=\sum_{m \geq 1} \frac{a_m}{p^m},
\]
where
\[
a_m:=
\begin{cases}
    p-1 & \text{if $m=\sum_{i=0}^r n_i$ for some $r \geq 0$} \\
    p-2 & \text{otherwise}
\end{cases}
\]
\end{theorem}

\begin{proof}
We take a perfectoid $R_\infty$ over $R$ as in \cite{BMPSTWW24}*{Lemma~4.23} and set
\[
\alpha_i:=\sup\{\alpha \in \Z[1/p]_{\geq 0} \mid p^\alpha\eta'_i \neq 0 \}, 
\]
where
\[
\eta'_i:=(H^d_\m(\var{R}) \to H^d_\m(\var{R}_\infty))(\eta_i)
\]
for every $i \geq 0$.
Then $\ppt(R;\div(p))=\alpha_0$ by \cref{test-perf-tests}.
\begin{claim}\label{cl:zenka}
We have
\[
\alpha_j=1-i_{n_j}+p^{-n_j}\alpha_{j+1}
\]
for every $j \geq 0$, where
\[
i_m:=\frac{1}{p}+\cdots+\frac{1}{p^m}
\]
for every integer $m\geq 1$.
\end{claim}
\begin{claimproof}
We fix $j \geq 0$ and set $n:=n_j$.
We take a system of $p$-power roots $\{p^{1/p^e}\}$ of $p$ in $R_\infty$ and $T \in R_\infty^\flat$ is the corresponding element of the projective system $\{p^{1/p^e}\}_{e \geq 0}$.
We consider the following commutative diagram;
\[
\begin{tikzcd}[row sep=0.3cm]
    \eta_j \arrow[r,phantom,"\in"] \arrow[d,mapsto]  &[0.01cm]
     H^d_\m(\var{R}) \arrow[r,"\Phi_{R,n}"] \arrow[d] & H^d_\m(Q_{R,n})  \arrow[d] & H^d_\m(F^n_*\var{R}) \arrow[d] \arrow[l,"V^{n-1}"'] &[0.01cm] \arrow[l,phantom,"\ni"]  \eta_{j+1} \arrow[d,mapsto]    \\
    \eta'_j \arrow[r,phantom,"\in"] & H^d_\m(\var{R}_\infty) \arrow[r,"\Phi_{R_{\infty},n}"] \arrow[rd,bend right=20,"\alpha"] & H^d_\m(Q_{R_\infty,n}) \arrow[d,"\var{\psi}_n"] & H^d_\m(F^n_*\var{R}_\infty)  \arrow[d,"F^{-n}"] \arrow[l,"V^{n-1}"'] & \arrow[l,phantom,"\ni"] \eta'_{j+1} \arrow[d,mapsto]  \\
     & & H^d_\m(\var{R}_\infty/(p^{i_n}))  & H^d_\m(\var{R}_\infty/(p^{1/p^n})), \arrow[l,"\cdot wp^{i_n-1}"'] & \arrow[l,phantom,"\ni"] F^{-n}(\eta'_{j+1}) 
\end{tikzcd}
\]
where $\alpha$ follows from the natural surjection and the commutativity of right bottom diagram follows from \cref{thm:purity-perfd}.
We consider the images of $\eta'_j$ and $ F^{-n}(\eta'_{j+1})$ in $H^{d+1}_\m(R_\infty)$, then we obtain
\begin{equation}\label{eq:loc-coh-ele}
    p^{(1-i_n)}(\eta'_j/p) = wp^{i_{n-1}}F^{-n}(\eta'_{j+1})/p^{i_n}=wF^{-n}(\eta'_{j+1})/(p^{1/p^n}).
\end{equation}
We consider the diagram
\[
\begin{tikzcd}
    H^{d+1}_\m(R_\infty) & \arrow[l] H^d_\m(\var{R}_\infty) \arrow[r,"\sim"] & H^d_{\m^\flat}(R_\infty^\flat/(T)) \arrow[r] & H^{d+1}_{\m^\flat}(R^\flat_\infty),
\end{tikzcd} 
\]
where $\m^\flat$ is the ideal of $R_\infty^\flat$ containing $T$ and $\m^\flat/(T) \simeq \m/(p)$, then the corresponding element to $\eta_j'/p$ and $\eta_{j+1}'/p$ are denoted by $\tau_j^\flat,\tau_{j+1}^\flat \in H^{d+1}_\m(R_\infty^\flat)$, respectively.
By (\ref{eq:loc-coh-ele}), we have 
\[
T^{(1-i_n)}\tau_j^\flat=w'F^{-n}(\tau_{j+1}^\flat)
\]
for some unit $w' \in R_\infty^\flat$.
Since $F^{-n}$ is an isomorphism, we have $T^\alpha\tau_{j+1}^\flat=0$ if and only if $T^{1-i_n+\alpha/p^n}\tau^\flat_j=0$ for $\alpha \in \Z[1/p]_{\geq 0}$.
In particular, we have $p^\alpha\eta'_{j+1}=0$ if and only if $p^{1-i_n+\alpha/p^n}\eta_j'=0$.
Therefore, we have
\[
\alpha_j=1-i_n+p^{-n}\alpha_{j+1},
\]
as desired.
\end{claimproof}\\
We write the $p$-adic expansion of $\alpha_j$ as
\[
\alpha_j=\sum_{m \geq 1}\frac{a_{j,m}}{p^m},
\]
where $a_{j,m} \in \{0,1,\ldots,p-1\}$ and the sum is not finite.
We note that since $1 \geq \alpha_j$, we can take such an expansion.
Since we have
\[
1-i_{n_j}=\frac{p-2}{p}+\cdots+\frac{p-2}{p^{n_j-1}}+\frac{p-1}{p^{n_j}},
\]
we obtain
\[
a_{jm}=
\begin{cases}
    p-2 & \text{if $m \leq n_j-1$} \\
    p-1 & \text{if $m=n_j$} \\
    a_{j+1,m} & \text{if $m \geq n_j+1$}
\end{cases}
\]
by \cref{cl:zenka}.
Therefore, we have
\[
a_{0,m}:=
\begin{cases}
    p-1 & \text{if $m=\sum_{i=0}^r n_i$ for some $r \geq 0$} \\
    p-2 & \text{otherwise},
\end{cases}
\]
as desired.
\end{proof}

\begin{theorem}\label{equiv-qfs-p-pure}\textup{(cf.~\cref{intro:ht-to-p-pure,CY-case-qfs})}
We use the notation introduced in \cref{notation:compare}
We assume $R$ is complete intersection.
Let $n$ a positive integer.
Then $\sht(R)=n$ if and only if $R$ is perfectoid pure and
\[
1-\frac{1}{p}-\cdots-\frac{1}{p^{n-1}} \geq \ppt(R;\div(p)) \geq 1- \frac{p+\cdots+p^{n-1}}{p^n-1}.
\] 
\end{theorem}

\begin{proof}
We take a non-zero element $\eta \in H^d_\m(\var{R})$ in the socle of $H^d_\m(\var{R})$.
To prove the ``only if'' direction, we assume that $\sht(R) = n$.
We construct a sequence $\{\eta_i\}_{i \geq 0}$ with $\eta_0=\eta$ and $\{n_i\}_{i \geq 0}$ using the procedure in \cref{F-W-program}.
Then we have $n_i \leq n$ for every $i \geq 0$ and $n_0=n$.
By \cref{const-seq}, we have
\[
\sum_{m \geq 1}\frac{p-2}{p^m}+\sum_{m \geq n} \frac{1}{p^m} \geq \ppt(R;\div(p)) \geq \sum_{m \geq 1} \frac{p-2}{p^m}+\sum_{r \geq 1}\frac{1}{p^{rn}}.
\]
By computation, we have
\[
1-\frac{1}{p}-\cdots-\frac{1}{p^{n-1}} \geq \ppt(R;\div(p)) \geq 1- \frac{p+\cdots+p^{n-1}}{p^n-1},
\] 
as desired.

Next, we prove the ``if part''.
Since we have
\[
\ppt(R;\div(p)) \geq 1- \frac{p+\cdots+p^{n-1}}{p^n-1}  > 1-i_{n},
\]
we obtain $\sht(R) \leq n$ by \cref{p-pure-to-qfs}.
We construct a sequence $\{\eta_i\}_{i \geq 0}$ with $\eta_0=\eta$ and $\{n_i\}_{i \geq 0}$ using the procedure in \cref{F-W-program}.
Therefore, we have $n_0=n$ by \cref{const-seq}, and in particular, we obtain $\Phi_{R,n}(\eta) \neq 0$ and $\Phi_{R,n-1}(\eta) \neq 0$ if $n \geq 2$.
Since $\eta$ is an element of the socle of $H^d_\m(\var{R})$, we have $\sht(R)=n$, as desired.
\end{proof}

\begin{corollary}\label{equiv-ffinfty}\textup{(cf.~\cref{intro:ht-to-p-pure})}
We use the notation introduced in \cref{notation:compare}
We assume $R$ is complete intersection.
Let $n$ a positive integer.
Then $\ppt(R;\div (p))=1-i_{n-1}$ if and only if $R$ is quasi-$(F,F^\infty)$-split of height $n$, where $i_{n-1}=1/p+\cdots+1/p^{n-1}$.
\end{corollary}

\begin{proof}
We take a non-zero element $\eta \in H^d_\m(\var{R})$ in the socle of $H^d_\m(\var{R})$.
To prove the ``if'' part, we assume that $R$ is quasi-$(F,F^\infty)$-split of height $n$.
We construct a sequence $\{\eta_i\}_{i \geq 0}$ with $\eta_0=\eta$ and $\{n_i\}_{i \geq 0}$ using the procedure in \cref{F-W-program}.
By the definition of quasi-$(F,F^\infty)$-split, we have $n_0=n$ and $n_i =1$ for every $i \geq 1$.
Thus, we obtain $\ppt(R;\div (p))=1-i_{n-1}$ by \cref{const-seq}, as desired.

To prove the ``only if' part, we assume $\ppt(R;\div (p))=1-i_{n-1}$, then $\sht(R)=n$ by \cref{equiv-qfs-p-pure}.
We construct a sequence $\{\eta_i\}_{i \geq 0}$ with $\eta_0=\eta$ and $\{n_i\}_{i \geq 0}$ using the procedure in \cref{F-W-program}.
By \cref{const-seq}, we have $n_0=n$ and $n_i = 1$ for $i \geq 1$.
Therefore, the ring $R$ is quasi-$(F,F^\infty)$-split of height $n$, as desired.
\end{proof}

\begin{example}\label{ex:threshold}
By \cref{equiv-ffinfty}, we can compute perfectoid pure thresholds for examples in \cref{ex:fedder} other than (1).
We will compute perfectoid pure threshold of \cref{ex:fedder} (1) in \cref{fermat-ell}.
\begin{enumerate}
    \item Let $R:=\Z_{(p)}[[x,y,z]]/(z^2+y^3+z^5)$, then 
    \[
    \ppt(R;\div(p))=
    \begin{cases}
        1/8 & \textup{if $p=2$}, \\
        5/9 & \textup{if $p=3$}, \\
        4/5 & \textup{if $p=5$}, \\
        1 & \textup{otherwise}.
    \end{cases}
    \]
    We note that if $p \geq 7$, then $\var{R}$ is $F$-pure.
    \item Let $R:=\Z_{(p)}[[x,y,z,w]]/(w^2+xyz(x+y+z))$ and $p=2$, then $\ppt(R;\div(p))=1/2$.
    \item Let $R:=\Z_{(p)}[[x,y,z,w]]/(w^2+xyz(x+y+z)+p(xy+xz+yz)w)$ and $p=2$, then $\ppt(R;\div(p))=1/4$.
    \item Let $R:=\Z_{(p)}[[x,y,z]]/(z^2+x^2y+xy^n)$ for $n \geq 2$ and $p=2$, then $\ppt(R;\div(p))=1/2^{\rup{\log_2 n}}$.
\end{enumerate}    
\end{example}

\begin{corollary}\label{qfs-reg-bcm-reg}\textup{(cf.~\cite{p-pure}*{Proposition~6.7})}
Let $(R,\m)$ be a Noetherian complete local domain with $p \in \m \backslash \{0\}$.
We assume $R$ is complete intersection.
If $R/pR$ is quasi-$F$-split and $R[1/p]$ is regular, then $R$ is BCM-regular.
\end{corollary}

\begin{proof}
If $R$ is not BCM-regular, then $(R,p^{\varepsilon}\div(p))$ is not perfectoid pure by the proof of \cite{p-pure}*{Proposition~6.7} for every $\varepsilon \in \Z[1/p]_{>0}$.
Therefore, $R$ is BCM-regular by \cref{equiv-qfs-p-pure}.
\end{proof}

\begin{corollary}\label{range-ppt}\textup{(\cref{intro:range-ppt})}
We use the notation introduced in \cref{notation:compare}
We assume $R$ is complete intersection and perfectoid pure.
If $\ppt(R;\div(p)) > \frac{p-2}{p-1}$, then there exists a positive integer $n$ such that $R$ is quasi-$n$-split, and in particular, we have
\[
1-\frac{1}{p}-\cdots-\frac{1}{p^{n-1}} \geq \ppt(R;\div(p)) \geq 1- \frac{p+\cdots+p^{n-1}}{p^n-1}.
\]    
\end{corollary}

\begin{proof}
Since we have
\[
\ppt(R;\div(p)) > \frac{p-2}{p-1}=\lim_{n \to \infty} (1-i_n),
\]there exists a positive integer $m$ such that $\ppt(R;\div(p)) > 1-i_m$.
By \cref{p-pure-to-qfs}, we have that $R$ is $m$-quasi-$F$-split, and in particular, there exists a positive integer $n$ such that $\sht(R)=n$.
By \cref{intro:ht-to-p-pure}, we have
\[
1-\frac{1}{p}-\cdots-\frac{1}{p^{n-1}} \geq \ppt(R;\div(p)) \geq 1- \frac{p+\cdots+p^{n-1}}{p^n-1},
\]  
as desired.
\end{proof}

\section{Case of graded rings}
In this section, we study quasi-$F$-splitting of graded rings.
The results vary greatly depending on an invariant called a-invariant.
\begin{notation}\label{notation:graded}
Let $S=\bigoplus_{i \in \Q_{\geq 0}} S_i$ be a Noetherian graded ring such that $S_0$ is a divisorial valuation ring with $p$ in the maximal ideal of $S_0$, $S$ is $p$-torsion free, and $S_i=0$ if $i \notin \Z$.
We set $\m:=(p, S_{>0})$, $R:=S_\m$, and $d:=\dim{\var{R}}$.
Then $R$ satisfies  conditions in \cref{notation:compare}.
Furthermore, we define $S^{(m)}:=\bigoplus_{i \geq 0} S_{im}$ as graded ring for every $m \in \Q_{>0}$.
For a graded $S$-module $M$ and $m \in \Q$, we define a graded $S$-module $M(m)$ by $M(m)_i=M_{i+m}$ for all $i \in \Q$. 
Since $H^d_\m(\var{R}) \simeq H^d_\m(\var{S})$, we endow $H^d_\m(\var{R})$ with a graded structure.
We set $a(\var{S}):=\max\{m \in \Z \mid H^{d}_\m(\var{S})_m \neq 0\}$.
Furthermore, we define the $\Q$-graded structure on $W_n(S)$ by 
\[
W_n(S)_i:=\{(a_0,\ldots,a_{n-1}) \in W_n(S) \mid a_0 \in S_i,\ldots, a_{n-1} \in S_{p^{n-1}i} \} 
\]
for $i \in \Q$ (cf.~\cite{KTTWYY2}*{Proposition~7.1}).
We consider the graded structure on $Q_{S,n}$ by $\var{W}_n(S)^{(p)}$.
Then we also endow $H^d_\m(Q_{R,n})$ with a graded structure.
\end{notation}

\subsection{Case of $a(\var{S})<0$}
In this subsection, we study the case where $a(\var{S})<0$, and the case is related to the cone of log Fano pairs (cf.~\cref{log-fano,log-fano-pair}).

\begin{proposition}\label{program-degree}
We use the notation introduced in \cref{notation:graded} and assume $R$ is quasi-$F$-split.
Let $\eta \in H^d_\m(\var{R})\backslash \{0\}$ be a homogeneous element of degree $\nu$.
Then we can construct sequences $\{\eta_i \in H^d_\m(\var{R}) \backslash \{0\} \}_{i \geq 0}$ and $\{n_i \in \Z_{\geq 1}\}_{i \geq 0}$ using a procedure in \cref{F-W-program} such that $\eta_i$ is a homogeneous element of degree $p^{\sum_{j=0}^{i-1} n_j} \nu$.  
\end{proposition}

\begin{proof}
We note that $\Phi_{S,n} \colon S \to Q_{S,n}$, $V^{n-1} \colon F^n_*\var{S}^{(p^n)} \to Q_{S,n}$ and $F \colon \var{S} \to F^e_*\var{S}^{(p)}$ are homomorphisms of graded $S$-modules.
We assume that we get a homogeneous element $\eta_{i}$ of degree $p^{e_{i}}\nu$ by a procedure in \cref{F-W-program} for an integer $i \geq 0$, where $e_i=\sum_{j=0}^{i-1}n_j$.
If $n_{i}=1$, then $F(\eta_{i})=\eta_{i+1}$ is a homogeneous element of degree $p^{e_{i}+1}\nu$, as desired.
If $n_{i} \geq 2$, then since $\Phi_{R,n_{i}}(\eta_i)$ is homogeneous, we can take $\eta_{i+1}$ as a homogeneous element.
Furthermore, the degree of $\eta_{i+1}$ is $p^{e_i+n_i}\nu$, as desired.
\end{proof}

\begin{theorem}\label{Fano}\textup{(\cref{intro:Fano})}
We use the notation introduced in \cref{notation:graded}.
We assume $\var{S}$ is normal quasi-Gorenstein and satisfies $a(\var{S}) < 0$.
We further assume that $\Spec{\var{S}} \backslash \{\m\var{S}\}$ is strongly $F$-regular.
If $R$ is quasi-$F$-split, then $R$ is perfectoid BCM-regular.
Furthermore, if $p=2$, the converse implication holds.
\end{theorem}

\begin{proof} 
We take a perfectoid BCM-algebra $B$ over $R$.
First, we prove that $H^d_\m(\var{R}) \to H^d_\m(\var{B})$ is injective.
We take a homogeneous element $\eta \in H^d_\m(\var{R})$ contained in the socle of $H^d_\m(\var{R})$.
Since $a(\var{S}) < 0$, the degree of $\eta$ is negative.
By \cref{program-degree}, we can take sequences $\{\eta_i\}_{i \geq 0}$ and $\{n_i\}_{i \geq 0} $ using a procedure in \cref{F-W-program} such that $\eta_i$ is homogeneous of degree $\nu_i:=p^{\sum_{j=0}^{i-1}}\nu$.
We set
\[
\eta_i':=\left(H^d_\m(\var{R}) \to H^d_\m(\var{B}) \right)(\eta_i).
\]
Since $\var{R}$ is $F$-pure outside of $\m\var{R}$, the kernel of $F\colon H^d_\m(\var{R}) \to H^d_\m(F_*\var{R})$ is of finite length, and in particular, there exists $i \geq 0$ such that $F^e(\eta_i) \neq 0$ for every $e \geq 1$.
By \cref{thm:fano-test-f-pure}, we have 
\[
\eta_i \notin 0^*_{H^d_\m(\var{R})}.
\]
Furthermore, since we have
\[
\mathrm{Ker}\left(H^d_\m(\var{R}) \to H^d_\m(\var{B})\right) \subseteq
0^*_{H^d_\m(\var{R})} 
\]
by the proof of \cite{bcm-reg}*{Proposition~5.3}, we obtain $\eta_i' \neq 0$.
Since $H^{d-1}_\m(\var{B})=0$, we obtain $\eta_j'  \neq 0$ for every $j \leq i$, inductively.
In particular, we have $\eta'_0 \neq 0$ and the map $H^d_\m(\var{R}) \to H^d_\m(\var{B})$ is injective, as desired.

Since $\var{S}$ is quasi-Gorenstein, so is $\var{R}$, and in particular, the homomorphism $\var{R} \to \var{B}$ is pure.
Since $\var{B}$ is a BCM-algebra over $\var{R}$, we obtain that $\var{R}$ is Gorenstein, thus so is $R$.
Since $H^d_\m(\var{R}) \to H^d_\m(\var{B})$ is injective and $B$ is a BCM-algebra over $R$, we obtain the injectivity of $H^{d+1}_\m(R) \to H^{d+1}_\m(B)$, thus the homomorphism $R \to B$ is pure, as desired.
Next, we assume $p=2$ and $R$ is BCM-regular.
Since $R$ is Gorenstein, $R$ is quasi-$F$-split by \cref{non-qfs-to-non-bcm}.
\end{proof}

\begin{corollary}\label{log-fano}
Let $k$ be an algebraically closed field of characteristic $p$.
Let $X$ be a normal integral projective flat scheme over $\Spec{W(k)}$.
We assume $-K_X$ is ample $\Q$-Cartier.
We set 
\[
S:=\bigoplus_{m \geq 0} H^0(X,\cO_X(-mK_X)).
\]
The closed fiber of $X \to \Spec{W(k)}$ is denoted by $\var{X}$ and we assume that $\var{X}$ is klt.
We assume one of the following conditions:
\begin{enumerate}
    \item $\var{X}$ is strongly $F$-regular and quasi-$F$-split, or
    \item $\dim{\var{X}}=2$ and $p > 5$.
\end{enumerate}
Then $S_{(p,S_{>0})}$ is perfectoid BCM-regular.
\end{corollary}

\begin{proof}
By \cite{KTTWYY2}*{Theorem~6.3} and \cite{Hara2}*{Theorem~1.1}, we may assume $\var{X}$ is strongly $F$-regular and quasi-$F$-split.
We set $d:=\dim(X)$ and
\[
S':=\bigoplus_{m \geq 0} H^0(\var{X},\cO_{\var{X}}(-mK_{\var{X}})).
\]
Then $S'$ is normal, quasi-Gorenstein by \cite{watanabe81}*{Theorem~2.8}, and quasi-$F$-split by \cite{KTTWYY2}*{Theorem~7.16} and \cite{kty}*{Proposition~2.25}.
Since we have
\[
H^d_\m(S')_m \simeq H^{d-1}(\var{X},\cO_{\var{X}}(-mK_{\var{X}})) \simeq H^0(\var{X},\cO_{\var{X}}((m+1)K_{\var{X}}))
\]
and $-K_{\var{X}}$ is ample, we obtain $a(\var{S'}) < 0$.
Furthermore, since $\var{X}$ is strongly $F$-regular, so is $\Spec{S'} \backslash \{S'_{>0}\}$.
Therefore, it is enough to show that $S/pS \simeq S'$ as graded rings by \cref{Fano}.
We note that if $S/pS \simeq S'$, then we have 
\[
H^0(X,\cO_X) \otimes_{W(k)} k \simeq S_0/pS_0 \simeq S'_0 \simeq H^0(\var{X},\cO_{\var{X}}) \simeq k,
\]
and in particular, $S_0$ is divisorial valuation ring with the maximal ideal $(p)$.
Furthermore, by \cite{bcm-reg}*{Theorem~6.27}, every stalk of $\cO_X$ is perfectoid BCM-regular, and in particular, the sheaf $\cO_{X}(-mK_X)$ is Cohen-Macaulay. 
Since we have $\cO_X(-mK_X) \otimes_{\cO_X} \cO_{\var{X}} \simeq \cO_{\var{X}}(-mK_{\var{X}})$ on the codimension one point of $\var{X}$ and $\cO_{X}(-mK_X)$ satisfies condition $S_3$ for each integer $m$, we obtain $\cO_X(-mK_X) \otimes_{\cO_X} \cO_{\var{X}} \simeq \cO_{\var{X}}(-mK_{\var{X}})$.
Since $\var{X}$ is quasi-$F$-split, we obtain $H^1(\var{X},\cO_{\var{X}}(-mK_{\var{X}}))=0$ for every integer $m \geq 0$ by \cite{KTTWYY1}*{Theorem~3.15}.
By \cite{hartshorne}*{Theorem~12.11}, we obtain
\[
H^0(X,\cO_X(-mK_X)) \otimes_{W(k)} k \simeq H^0(\var{X},\cO_{\var{X}}(-mK_{\var{X}})),
\]
thus, we obtain $S/pS \simeq S'$, as desired.
\end{proof}

\begin{corollary}\label{log-fano-pair}
Let $k$ be an algebraically closed field of characteristic $p$.
Let $X$ be a normal integral projective flat scheme over $\Spec{W(k)}$ and $\Delta$ an effective $\Q$-Weil divisor on $X$, whose components are flat over $\Spec{W(k)}$.
We assume that $-(K_X+\Delta)$ is ample $\Q$-Cartier and $\Delta$ has standard coefficients.
We set 
\[
S:=\bigoplus_{m \geq 0} H^0(X,\cO_X(-m(K_X+\Delta))).
\]
The closed fiber of $(X,\Delta) \to \Spec{W(k)}$ is denoted by $(\var{X},\var{\Delta})$ and we assume that $(\var{X},\var{\Delta})$ is strongly $F$-regular, the Picard rank of $\var{X}$ is one, and  $\var{X}$ is $\Q$-factorial.
We assume one of the following conditions:
\begin{enumerate}
    \item $(\var{X},\var{\Delta})$ is quasi-$F$-split,
    \item $\dim{\var{X}}=1$, or
    \item $\dim{\var{X}}=2$ and $p > 41$.
\end{enumerate}
Then $S_{(p,S_{>0})}$ is perfectoid BCM-regular.
\end{corollary}

\begin{proof}
By \cite{KTTWYY1}*{Corollary~5.16} and \cite{KTTWYY2}*{Theorem~C}, we have $(\var{X},\var{\Delta})$ is quasi-$F$-split.
By the same argument as the proof of \cref{log-fano}, it is enough to show that for every integer $m \geq 0$,
\begin{enumerate}
    \item[(a)] $\cO_{X}(-m(K_{X}+\Delta)) \otimes_{\cO_X} \cO_{\var{X}} \simeq \cO_{\var{X}}(-m(K_{\var{X}}+\var{\Delta}))$,
    \item[(b)] $H^1(\var{X},\cO_{\var{X}}(-m(K_{\var{X}}+\var{\Delta})))=0$, and
    \item[(c)] $H^{d-1}(\var{X},\cO_{\var{X}}(-m(K_{\var{X}}+\var{\Delta})))=0$, where $d:=\dim{X}$.
\end{enumerate}
Indeed, the cone of $(\var{X},\var{\Delta})$ coincides with $\var{S}$ by (a) and (b), the $a$-invariant $a(\var{S})$ is negative by (c).
By an argument in the proof of \cref{log-fano}, we have that $X$ is divisorially Cohen-Macaulay, that is, for all divisorial sheaves on $X$ are Cohen-Macaulay.
We set $\Delta:=\sum a_i E_i$, where each $E_i$ is a prime divisor and $a_i$ has a form $a_i=(m_i-1)/m_i$ for some integer $m_i \geq 2$.
Since $(\var{X},\var{\Delta})$ is strongly $F$-regular, we have $\rdown{\var{\Delta}}=0$, and in particular, we have $E_i|_{\var{X}}$ is a reduced divisor.
Therefore, we obtain that $\cO_{X}(-m(K_X+\Delta)) \otimes_{\cO_X} \cO_{\var{X}} \simeq \cO_{\var{X}}(-m(K_{\var{X}}+\var{\Delta}))$ on codimension one points of $\var{X}$.
Since $X$ is divisorially Cohen-Macaulay, we obtain  condition (a).
Next, we prove the holding condition (b).
We take an integer $m \geq 0$ and set $L:=\rdown{-m(K_X+\Delta)}$, which is $\Q$-Cartier since $\var{X}$ is $\Q$-factorial.
Since the Picard rank of $\var{X}$ is one and $-K_X$ is ample, we obtain that $L-K_X$ or $-L$ is ample.
By \cite{KTTWYY1}*{Theorem~5.1}, we obtain  condition (b) if $\dim{\var{X}} \geq 2$.
Therefore, it is enough to show  condition (c).
By the Serre duality, we obtain
\[
H^{d-1}(\var{X},\cO_{\var{X}}(-m(K_{\var{X}}+\var{\Delta}))) \simeq H^0(\var{X},\cO_{\var{X}}((m+1)K_{\var{X}}+\rup{m\var{\Delta}})).
\]
Since $-(K_{\var{X}}+\var{\Delta})$ is ample, it is enough to show that $\rup{m\var{\Delta}} \leq (m+1)\var{\Delta}$ for every integer $m \geq 0$.
We take a prime divisor $E$ on $\var{X}$ and a positive integer $n$ with $\frac{n-1}{n}=\mathrm{ord}_E(\var{\Delta})$.
We take integers $q \geq 0$ and $n > r \geq 0$ such that $m=nq+r$.
Then we have 
\[
\rup{m\frac{n-1}{n}}=m-\rdown{\frac{m}{n}}=m-q.
\]
On the other hand, we have
\[
(m+1)\frac{n-1}{n}=m+1-q-\frac{1}{n},
\]
thus we obtain $(m+1)\frac{n-1}{n} \geq \rup{m\frac{n-1}{n}}$, as desired.
\end{proof}

\begin{corollary}\label{example-new}\textup{(\cref{intro:example-new})}
There exists a Noetherian graded ring $S=\bigoplus_{i \in \Z_{\geq 0}}S_i$ such that $S$ satisfies the following properties.
\begin{enumerate}
    \item $S$ is torsion free and $S_0$ is a divisorially valuation ring with maximal ideal $(p)$.
    \item We set $\m:=(S_{>0},p)$ and $R:=S_\m$,
    then $R$ is perfectoid BCM-regular.
    \item $R/pR$ is normal Gorenstein but not strongly $F$-regular.
\end{enumerate}
\end{corollary}

\begin{proof}
We set $k:=\overline{\F}_p$.
If $p=2,3$ or $5$, then we consider the pair 
\[
(\P^1_{W(k)},\Delta:=\frac{1}{2}E_1+\frac{2}{3}E_2+\frac{4}{5}E_3),
\]
where $E_1,E_2,E_3$ are disjoint sections of $\P^1_{W(k)} \to \Spec{W(k)}$.
Then it satisfies the assumptions of \cref{log-fano-pair}, thus the localization of 
\[
S:=\bigoplus_{m \geq 0} H^0(X,\cO_X(-mK_X))
\]
is perfectoid BCM-regular.
On the other hand, we see that $S=W(k)[x,y,z]/(z^2+x^3+y^5)$, thus $S/pS$ is not strongly $F$-regular by \cite{Hara2}*{Theorem~1.1}.
Therefore, the ring $S$ satisfies the desired conditions.
Next, we assume $p \geq 7$.
By \cite{CTW18}*{Theorem~1.1}, there exists a projective klt surface $\var{X}$ such that $-K_{\var{X}}$ is ample and $\var{X}$ is not globally $F$-regular.
By \cite{CTW18}*{Theorem~1.1} and \cite{ABL}*{Proposition~2.5}, a log resolution of $\var{X}$ lifts to $W(k)$.
By \cite{BBKW}*{Proposition~6.2}, there exists a flat normal integral projective scheme $X$ over $\Spec{W(k)}$ such that $X \times_{\Spec{W(k)}} \Spec{k} \simeq \var{X}$ and $K_X$ is $\Q$-Cartier.
Since $-K_{\var{X}}$ is ample, so is $-K_X$.
By \cref{log-fano},  the localization $R$ of
\[
S:=\bigoplus_{m \geq 0} H^0(X,\cO_X(-mK_X))
\]
is perfectoid BCM-regular.
Furthermore, by the proof of \cref{log-fano}, $S/pS$ is cone of $\var{X}$, thus $S/pS$ is not strongly $F$-regular by \cite{smith}*{Theorem~3.10}.
Thus, $R$ is perfectoid BCM-regular, but $R/pR$ is not strongly $F$-regular.
Moreover, the ring $S_0$ is the divisorially valuation ring with maximal ideal $(p)$, thus $S$ satisfies the desired conditions.
\end{proof}

\subsection{Case of $a(\var{S}) \geq 0$}
In this subsection, we study the case where $a(\var{S}) \geq 0$.

\begin{proposition}\label{qfs-to-non-general}
We use the notation introduced in \cref{notation:graded}.
If $R$ is quasi-$F$-split, then we have $a(R)\leq 0$.
\end{proposition}

\begin{proof}
Suppose $a(R) >0$ and $R$ is quasi-$F$-split, then there exists a homogeneous non-zero element $\wt{\eta} \in H^{d+1}_\m(R)$ of positive degree.
Replacing $\wt{\eta}$ by $p^s\wt{\eta}$ for some $s \geq 0$, we may assume $\wt{\eta}$ is $p$-torsion free.
Then there exists a homogeneous element $\eta \in H^d_\m(\var{R})$ of positive degree such that $\wt{\eta}=\eta/p$.
We take a sequence $\{\eta_i\}_{i \geq 0}$ as in \cref{program-degree}, then the degrees of $\eta_i$ are strictly increasing, thus it contradicts $\eta_i \neq 0$ for every $i \geq 0$, as desired.
\end{proof}

\begin{proposition}\label{CY-case}
We use the notation introduced in \cref{notation:graded}. 
Assume that $R$ is Gorenstein and that $a(\var{S}) = 0$.
Then $\sht(R) = \sht(\var{R})$.
\end{proposition}

\begin{proof}
By \cref{compsre-posi-case}, we have $\sht(R) \leq \sht(\var{R})$.  
Let $n := \sht(R)$. Then it suffices to show that $\var{R}$ is $n$-quasi-$F$-split.  
If $n = 1$, the assertion follows immediately from the isomorphisms $Q_{R,1} \simeq Q_{\var{R},1} \simeq F_*\var{R}$.  
Hence, we may assume $n \geq 2$.

Let $\eta \in H^d_\m(\var{R})$ be a homogeneous nonzero element of degree zero.  
Since $a(\var{S}) = 0$, the socle of $H^d_\m(\var{R})$ is concentrated in degree zero, and thus $\eta$ lies in the socle.

We consider the following commutative diagram:
\[
\begin{tikzcd}
    H^d_\m(\var{R}) \arrow[r,"\Phi_{R,n}"] \arrow[rd,"\Phi_{\var{R},n}"'] & H^d_\m(Q_{R,n}) \arrow[d] & \arrow[l,"V^{n-1}"'] H^d_\m(F^n_*\var{R}) \arrow[d,"\alpha"] \\
    & H^d_\m(Q_{\var{R},n}) & \arrow[l,"V^{n-1}"'] H^d_\m(F^{n-1}_*B_{\var{R}}) 
\end{tikzcd}
\]

Since $\Phi_{R,n}(\eta) \neq 0$ and $\Phi_{R,n-1}(\eta) = 0$, and $R$ is Cohen--Macaulay, there exists a unique $\eta' \in H^d_\m(\var{R})$ such that $\Phi_{R,n}(\eta) = V^{n-1}(\eta')$.  
By \cref{program-degree}, the degree of $\eta'$ is zero.

Moreover, since $\var{R}$ is not $F$-pure, the map $H^d_\m(\var{R})_0 \to H^d_\m(F_*\var{R})_0$ is the zero map.  
In particular, $\eta'$ is not contained in the image of Frobenius.  
Thus, under the exact sequence
\[
0 \to F^{n-1}_*\var{R} \xrightarrow{F^{n-1}_*F} F^n_*\var{R} \to F^{n-1}_*B_{\var{R}} \to 0,
\]
the element $\eta'$ does not map to zero under the induced map $\alpha$ on local cohomology.  
Hence, to prove that $\eta'$ survives in $H^d_\m(Q_{\var{R},n})$, it suffices to show that the map
\[
V^{n-1}_0 \colon H^d_\m(F^{n-1}_*B_{\var{R}})_0 \to H^d_\m(Q_{\var{R},n})_0
\]
is injective.

We now prove that $F_0 \colon H^d_\m(W_m(\var{R}))_0 \to F_*H^d_\m(W_m(\var{R}))_0$ is the zero map for every $1 \leq m \leq n-1$ by induction on $m$.

For the base case $m=1$, the assertion follows from the fact that $\var{R}$ is not $F$-pure.  
Assume the claim holds for $m-1$ with $m \leq n-1$.

Since $\var{R}$ is not $(n-1)$-quasi-$F$-split, the composition
\[
H^d_\m(W_m(\var{R}))_0 \xrightarrow{F} F_*H^d_\m(W_m(\var{R}))_0 \to H^d_\m(Q_{\var{R},m})_0
\]
is zero.

From the short exact sequence
\[
0 \to F_*W_{m-1}(\var{R}) \xrightarrow{VF} F_*W_m(\var{R}) \to Q \to 0,
\]
and using the induction hypothesis that $F_0$ on $H^d_\m(W_{m-1}(\var{R}))_0$ is zero, we find that the map $F_*H^d_\m(W_m(\var{R}))_0 \to H^d_\m(Q_{\var{R},m})_0$ is injective.  
Thus, $F_0$ on $H^d_\m(W_m(\var{R}))_0$ must be zero.

Next, consider the following commutative diagram, where each row is a short exact sequence:
\[
\begin{tikzcd}
0 \arrow[r] & H^d_\m(F^{n-2}_*\var{R})_0 \arrow[r, "V^{n-2}"] \arrow[d, "\cong"] & H^d_\m(W_{n-1}(\var{R}))_0 \arrow[r] \arrow[d, "\cong"] & H^d_\m(W_{n-2}(\var{R}))_0 \arrow[r] \arrow[d, "\cong"] & 0 \\
 & H^d_\m(F^{n-2}_*B_{\var{R}})_0 \arrow[r, "V^{n-2}"] & H^d_\m(B_{\var{R},n-1})_0 \arrow[r] & H^d_\m(B_{\var{R},n-2})_0 \arrow[r] & 0.
\end{tikzcd}
\]

By the above, each vertical map is an isomorphism.  
Therefore, we obtain the injection:
\[
H^d_\m(F^{n-2}_*B_{\var{R}})_0 \hookrightarrow H^d_\m(B_{\var{R},n-1})_0.
\]

Furthermore, from the exact sequence
\[
0 \to B_{\var{R},n-1} \xrightarrow{V} Q_{\var{R},n} \to \var{R} \to 0,
\]
and the fact that $\var{R}$ is Cohen--Macaulay, we deduce that the map
\[
H^d_\m(B_{\var{R},n-1}) \to H^d_\m(Q_{\var{R},n})
\]
is injective.

Combining the above, we obtain the desired injectivity:
\[
V^{n-1}_0 \colon H^d_\m(F^{n-1}_*B_{\var{R}})_0 \hookrightarrow H^d_\m(Q_{\var{R},n})_0.
\]
This completes the proof.
\end{proof}

\begin{theorem}\label{CY-case-qfs}\textup{(\cref{intro:CY-case})}
We use Notation \cref{notation:graded}.
We assume $R$ is complete intersection and $a(\var{S})=0$.
If $\sht(\var{R})=n < \infty$, then $R$ is perfectoid pure and we have
\[
\ppt(R;\div(p))=1-\frac{p+\cdots+p^{n-1}}{p^n-1}.
\]
\end{theorem}

\begin{proof}
By \cref{CY-case}, we obtain $\sht(R)=n$.
By \cref{thm:qfs-to-p-pure-comp}, the ring $R$ is perfectoid pure.
Let $\eta \in H^d_\m(\var{R})$ be a non-zero element of the socle of $H^d_\m(\var{R})$, then $\deg(\eta)=0$.
Using a procedure as in \cref{program-degree}, we construct a sequence $\{\eta_i\}_{i \geq 0}$ with $\eta_0=\eta$ and $\{n_i\}_{i \geq 0}$, then the degree of $\eta_i$ is zero for every $i \geq 0$, and in particular, each $\eta_i$ is an element of the socle of $H^d_\m(\var{R})$.
Since $\sht(R)=n$, we have $n_i=n$ for every $i \geq 0$.
By \cref{const-seq}, we have 
\[
\ppt(R;\div(p))=1-\frac{p+\cdots+p^{n-1}}{p^n-1},
\]
as desired.
\end{proof}

\begin{proposition}\label{fermat-ell}
Let $R:=\mathbb{Z}_{p}[[x,y,z]]/(x^3+y^3+z^3)$ and $p \equiv 2 \mod 3$.
\begin{enumerate}
    \item The ring $R$ is perfectoid pure and $\ppt(R;\div(p))=1-p/(p^2-1)$. 
    \item The ring $R \otimes_{\Z_{p}} R$ is not perfectoid pure.
\end{enumerate}
\end{proposition}

\begin{proof}
We note that $\sht(\var{R})=2$.
The first assertion follows from \cref{CY-case-qfs}.
We use same notations $R_\infty$, $\eta \in H^2_\m(\var{R})$, $\tau \in H^2_\m(\var{R}_\infty)$, and $\tau^\flat \in H^3_{\m^\flat}(R^\flat_\infty)$ in the proof of \cref{CY-case-qfs}.
By the proof of \cref{CY-case-qfs}, we obtain
\begin{equation}\label{eq:ell}
    T^{(1-i_n)}\tau^\flat=w'F^{-2}(\tau^\flat).
\end{equation}
Suppose $R \otimes_{\Z_{(p)}} R$ is perfectoid pure, then so is $R':=(R \otimes_{\Z_{p}}R)^{\wedge \n}$ by \cite{p-pure}*{Lemma~4.8}, where $\n$ be the  ideal of $R \otimes_{\Z_{p}} R$ generated by $p$, $\m \otimes R$, and $R \otimes \m$.
We note that 
\[
R'\simeq \Z_{p}[[x,y,z,x',y',z']]/(x^3+y^3+z^3,x'^3+y'^3+z'^3).
\]
We take a surjection $S':=\Z_{p}[[x,y,z,x',y',z']] \to R'$ and we define perfectoids $S'_\infty$ and $R'_\infty:=R^{S'_\infty}_{perfd}$ as in \cite{p-pure}*{Section~4.3}.
By the natural homomorphism $H^3_\m(R) \otimes_{\Z_{p}} H^3_\m(R) \to H^5_\n(R')$, the image of the element $[z^2/xyp] \otimes [z'^2/x'y'p]$ is
\[
[z^2z'^2/xyx'y'p],
\]
which a non-zero element of the socle of $H^5_\n(R')$.
Since $R'$ is perfectoid pure, the image $\eta'$ in $H^5_\m(R')$ is non-zero.
We set
\[
\tau':=(H^5_\n(R') \to H^5_\n(R'_\infty))(\eta')
\]
and denote $\tau'^\flat \in H^5_\n((R')_\infty^\flat)$ corresponding to $\tau'$.
Then we have
\[
T^{2(1-i_n)}\tau'^\flat=w''F^{-2}(\tau'^\flat)
\]
for some unit $w'' \in (R')_\infty^\flat$ by (\ref{eq:ell}).
In particular, we have $T^{1/p}\tau'^\flat \neq 0$, thus we have $\ppt(R';\div(p)) \geq 1/2$.
On the other hand, since $R'$ is not quasi-$F$-split by \cref{ex:non-qfs}, thus $\ppt(R;\div(p))=0$ by \cref{p-pure-to-qfs}.
Therefore, we obtain a contradiction.
\end{proof}

\appendix
\section{Construction of functorial test perfectoids}\label{app:A}
In this section, we prove the existence of the functorial test perfectoid (\cref{thm:const-test-perfd}).
\begin{theorem}\label{thm:const-test-perfd}
There exists a functor $T$ from the category of $\Z_{(p)}$-algebras to the category of perfectoids such that $T(R)$ is a test perfectoid over $R$ and $T(R)$ has a compatible system of $p$-power roots.
If $R$ is $p$-torsion free, then so is $T(R)$.
\end{theorem}

\begin{proof}
First, we construct a perfectoid $T(R)$ over a $\Z_{(p)}$-algebra $R$.
We set $A:=\Z_{p}[X]^{\wedge (p,d)}$ and $d:=X-p$, and we define the $\Z_p$-algebra homomorphism $\phi \colon A \to A$ by $\phi(X)=X^p$, then $(A,d)$ is a prism.
We define $A_\infty$ by the $(p,d)$-adic completion of perfection of $A$.
Let $A\{R\}:=A\{X_f \mid f \in R\}$ be a free $\delta$-ring as in \cite{BS}*{Lemma~2.11}, then $(A\{R\}^{\wedge (p,d)},d)$ is a prim.
Taking a perfection and $(p,d)$-completion, we obtain a perfect prism $(A\{R\}_\infty,(d))$.
We set $Z\{R\}_\infty:=A\{R\}_\infty/(d)$, then we have the $\Z_{(p)}$-algebra homomorphism $\Z_{(p)}[R]:=\Z_{(p)}[X_f \mid f \in R] \to A\{R\}_\infty$ defined by $X_f \mapsto X_f$.
We set 
\[
T(R):=(Z\{R\}_\infty \otimes_{\Z_{(p)}[R]} R)_{perfd},
\]
where $\Z_{(p)}[R] \to R$ is defined by $X_f \mapsto f$.
We note that $Z\{R\}_\infty$ is  a perfectoid and $\Z_{(p)}[R] \to R$ is surjective, the perfectoidazation exists by \cite{BS}*{Corollary~7.3}.
Furthermore, since the images of $\{\phi^{-e}(X)\}$ by the map $A\{R\}_\infty \to T(R)$ form a compatible system of $p$-power roots of $p$.

Next, we prove that $T(R)$ is a test perfectoid over $R$.
Let $B$ be a perfectoid $R$-algebra, then by Andre's flatness lemma \cite{BS}*{Theorem~7.14}, we may assume that $B$ contains a compatible system of $p$-power roots $\{p^{1/p^e}\}$ of $p$.
Let $(\wt{B},(d'))$ is a perfect prism such that $\wt{B}/(d') \simeq B$.
We consider the map $A_\infty=\Z_{(p)}[X^{1/p^{\infty}}]^{\wedge (p,d)} \to B$ by $X^{1/p^e} \mapsto p^{1/p^e}$, then it induces the homomorphism of perfectoids $A_{\infty}/(d) \to B$, and in particular, we obtain the homomorphism of perfect prisms $(A_\infty,(d)) \to (\wt{B},(d'))$ by \cite{BS}*{Theorem~3.10}.
For each element $f \in R$, we take a lift $\wt{f} \in \wt{B}$ of the image of $f$ by $R \to B$, then we can define the homomorphism of perfect prisms $(A\{R\}_\infty,(d)) \to (\wt{B},(d'))$ defined by $X_f \mapsto \wt{f}$, thus we obtain the homomorphism $Z\{R\}_\infty \to B$ and the commutative diagram
\[
\begin{tikzcd}
    \Z_{(p)}[R] \arrow[r] \arrow[d] & Z\{R\}_\infty \arrow[d]\\
    R \arrow[r] & B.
\end{tikzcd}
\]
In conclusion, we obtain an $R$-algebra homomorphism $T(R) \to B$, as desired.

Next, for a $\Z_{(p)}$-algebra homomorphism $\varphi \colon R \to S$, we define a ring homomorphism $T(\varphi) \colon T(R) \to T(S)$.
We define $Z\{R\}_\infty$ and $Z\{S\}_\infty$ as above, then we obtain the $\Z_{(p)}$-algebra homomorphism $Z\{R\}_\infty \to Z\{S\}_\infty$ by $X_f \mapsto X_{\varphi(f)}$.
Then it induces the homomorphism $T(\varphi) \colon T(R) \to T(S)$.
Then this $T$ defines the functor from the category of $\Z_{(p)}$-algebras to the category of perfectoids, as desired.

Finally, we prove that $T(R)$ is $p$-torsion free if $R$ is $p$-torsion free.
Since the endomorphism $\phi \colon A\{R\} \to A\{R\}$ is $(p,d)$-completely faithfully flat, so is  $A\{R\} \to A\{R\}_\infty$.
Since $A[R] \to A\{R\}$ is $(p,d)$-completely faithfully flat, so is $A[R] \to A\{R\}_\infty$.
Therefore, we obtain that $\Z_{(p)}[R] \to Z\{R\}$ is $p$-completely faithfully flat, and in particular, so is $R \to Z\{R\} \otimes_{Z_(p)[R]} R$.
Since $R$ is $p$-torsion free, so is $Z\{R\} \otimes_{Z_{(p)}[R]} R$.
By \cite{MSTWW}*{Lemma~A.2}, we obtain that $T(R)$ is $p$-torsion free, as desired.
\end{proof}

\section{Test ideal of graded ring}
In this section, we prove \cref{thm:fano-test-f-pure}, which is a technical result to use the proof of \cref{Fano}.
The proof of \cref{thm:fano-test-f-pure} is taught by Kenta Sato.

\begin{theorem}\label{thm:fano-test-f-pure}
Let $S:=\bigoplus_{m \in \Z_{\geq 0}}S_m$ be a Noetherian graded ring such that $S_0$ is a field of characteristic $p>0$.
We assume $S$ is normal, quasi-Gorenstein and $a(S)< 0$ and set $\m:=S_{>0},d:=\dim{S}$.
We further assume $\Spec{S} \backslash \{\m\}$ is strongly $F$-regular.
Then we have
\[
0^*_{H^d_\m(S)}=\cap_{e \geq 0} \Ker(H^d_\m(S) \xrightarrow{F^e} H^d_\m(F^e_*S)).
\]
\end{theorem}

\begin{proof}
Since $S$ is quasi-Gorenstein, we obtain $\omega_S \simeq S(\nu)$ as graded $S$-modules, then $\nu <0$ by $a(S)<0$.
By the isomorphism, the trace map $F_*\omega_S^{(p)} \to \omega_S$ is corresponding to $F_*(S(\nu))^{(p)} \to S(\nu)$, which is denoted by $u$.
By the Matlis duality, it is enough to show that $\tau(S)=\bigcap_{e \in \Z_{\geq 0}}\Im(u^e \colon F^e_*S \to  S)$, where $\tau(S)$ is the test ideal of $S$.
First, we prove the following claim.
\begin{claim}\label{cl:app-b}
We have $u^e(F^e_*S_{n})=0$ if $n \not\equiv \nu \mod p^e$ for every positive integer $e$.
\end{claim}
\begin{claimproof}
Since $u^e \colon F^e_*(S(\nu)^{(p^e)}) \to S(\nu)$ is a morphism of graded $S$-modules, for an integer $m \geq 0$, 
\[
(u^e)^{-1}(S_m)=(u^e)^{-1}(S(\nu)_{-\nu+m}) \subseteq F^e_*(S(\nu))^{(p^e)}_{-\nu+m}=F^e_*S_{(1-p^e)\nu+p^em}.
\]
In particular, if $u^e(F^e_*a) \neq 0$ for a homogeneous element $a \in S$, then $\deg(a) \equiv \nu \mod p^e$, as desired.
\end{claimproof}\\
Next, we prove the following claim.
\begin{claim}\label{app-b-2}
Let $\fq \subseteq S$ be a $\m$-primary ideal and a rational number $1>t>0$.
Then there exists a positive integer $e_0$, we have $u^e(F^e_*(\fq \cdot \m^{\rup{t(p^e-1)}}))=u^e(F^e_*S)$ for $e \geq e_0$.
\end{claim}
\begin{claimproof}
We may assume $\fq=\m^n$ for some positive integer $n$.
We take a rational number $1>t>0$.
We take $e_0$ such that $n+t(p^{e_0}-1) \leq p^{e_0}+\nu-1$ and $e \geq e_0$.
Since $\m^l \subseteq S_{\geq l}$ for every positive integer $l$, we have
\begin{align*}
    u^e(F^e_*\m^{n+\rup{t(p^e-1)}}) \subseteq u^e(F^e_*S_{\geq n+\rup{t(p^e-1)}}).
\end{align*}
Thus, it is enough to show that $u^e(F^e_*S_{\leq n+\rup{t(p^e-1)}})=0$.
We have
\[
u^e(F^e_*S_{\leq n+\rup{t(p^e-1)}}) \overset{(\star_1)}{\subseteq} u^e(F^e_*S_{\leq p^e+\nu-1}) \overset{(\star_2)}{=} 0,
\]
where $(\star_1)$ follows from the choice of $e_0$ and $(\star_2)$ follows from \cref{cl:app-b}, as desired.
\end{claimproof}\\
Next, we prove the following claim:
\begin{claim}\label{cl:app-b-3}
For every rational number $1 > t > 0$, we have 
\[
\sigma(S,\m^t)=\sigma(S)=u^e(F^e_*S),
\]
where $\sigma(S,\m^t)$ and $\sigma(S)$ are defined in \cite{FST}.
\end{claim}
\begin{claimproof}
By \cite{FST}*{Remark~14.5(1)}, we obtain $\sigma(S)=u^e(F^e_*S)$ for large enough $e$.
Furthermore, for enough  large $e$, we have
\begin{align*}
    \sigma(S) &\supseteq \sigma_1(S,\m^t) = \sum_{e' \geq 1} u^{e'}(F^{e'}_*\var{\m^{\rup{t(p^{e'}-1)}}}) \\
    &\supseteq  u^e(F^e_*\m^{\rup{t(p^e-1)}}) \overset{(\star_3)}{=} u^e(F^e_*S),
\end{align*}
where $(\star_3)$ follows from \cref{app-b-2}.
Therefore, we have $\sigma_1(S,\m^t)=\sigma(S)$.
Furthermore, since $S$ is $F$-pure outside of $\m$, the ideal $\sigma(S)$ is $\m$-primary.
Thus, for large enough $e$, we have
\begin{align*}
    \sigma(S) &=\sigma_2(S) \supseteq \sigma_2(S,\m^t) = \sum_{e' \geq 1} u^{e'}(F^{e'}_*\sigma_1(S,\m^t)\var{\m^{\rup{t(p^{e'}-1)}}}) \\
    &\supseteq u^e(F^e_*\sigma_1(S,\m^t)\m^{\rup{t(p^e-1)}}) \overset{(\star_4)}{=}u^e(F^e_*S)=\sigma(S),
\end{align*}
where $(\star_4)$ follows from \cref{app-b-2}, thus we have $\sigma_2(S,\m^t)=\sigma_1(S,\m^t)=\sigma(S)$.
Therefore, we have $\sigma(S,\m^t)=\sigma(S)$, as desired.
\end{claimproof}\\
Finally, we prove the assertion.
We take rational numbers $1 >t > \varepsilon > 0$, then
we have
\[
\tau(S) \subseteq \sigma(S) \overset{(\star_5)}{=} \sigma(S,\m^t) \overset{(\star_6)}{\subseteq} \tau(S,\m^{t-\varepsilon}) \subseteq \tau(S), 
\]
where $(\star_5)$ follows from \cref{cl:app-b-3} and $(\star_6)$ follows from \cite{FST}*{Proposition~14.10(4)} and the fact that $S$ is strongly $F$-regular outside of $\m$.
Therefore, we obtain that
\[
\tau(S)=\sigma(S)=\bigcap_{e \in \Z_{\geq 0}}\Im(u^e \colon F^e_*S \to  S),
\]
as desired.
\end{proof}

\section{Proof of Fedder-type criterion}\label{proof-fedder}
In this section, we provide a proof of \cref{Fedder}.
The proof of \cref{Fedder} is almost identical to a proof of \cite{TWY}*{Theorem~A}, however, for the convenience of the reader, we provide a proof.
\begin{notation}\label{notation:Fedder}
Let $(A,\m)$ be a regular local ring with finite ring homomorphism $\phi \colon A \to A$ such that $\phi$ is a lift of Frobenius.
We assume $p \in \m \backslash \m^2$ and set $\var{A}:=A/pA$, then $\var{A}$ is regular.
We fix a generator $u \in \Hom_{\var{A}}(F_*\var{A},\var{A})$.
Let $f_1,\ldots,f_r$ be a regular sequence in $A$ and we set $I:=(f_1,\ldots,f_r)\var{A}$ and $f:=f_1 \cdots f_r$.
We assume $R:=A/(f_1,\ldots,f_r)$ is $p$-torsion free and set $\var{R}:=R/pR$.
By \cref{decomposition-W}, we obtain the isomorphism
\begin{equation}\label{eq:isom}
\Hom_A(Q_{A,n},\var{A}) \simeq \Hom_A(F^{n}_*\var{A},\var{A}) \oplus \cdots \oplus \Hom_A(F_*\var{A},\var{A}) \overset{(\star_1)}{\simeq} F_*^{n}\var{A} \oplus \cdots \oplus F_*\var{A},
\end{equation}
where $(\star_1)$ is given by 
\[
F^e_*\var{A} \xrightarrow{} \Hom_A(F^e_*\var{A},\var{A})\ ;\ F^e_*a \mapsto (F^e_*b \mapsto u^e(F^e_*(ab))) 
\]
for each $n \geq e \geq 1$.
The homomorphism corresponding to $(F^n_*g_1,\ldots,F_*g_n)$ is denoted by $\varphi_{(g_1,\ldots,g_{n})}$.
\end{notation}

\begin{lemma}\label{lift-splitting}
We use the notation introduced in \cref{notation:Fedder}
Let $n$ be a positive integer.
Then $R$ is $n$-quasi-$F$-split if and only if there exists an $\var{A}$-module homomorphism $\psi \colon Q_{A,n} \to \var{A}$ such that 
\begin{enumerate}
    \item[(a)] $\psi(1)$ is a unit and
    \item[(b)] $\psi(\Ker(Q_{A,n} \to Q_{R,n})) \subseteq I$.
\end{enumerate}
\end{lemma}

\begin{proof}
Since $\phi$ is finite and $A$ is regular, the module  $\phi^e_*A$ is a free $A$-module for each positive integer $e$.
By \cref{decomposition-W}, the $A$-module $\phi_*W_n(A)$ is also free, and in particular, the $\var{A}$-module $Q_{A,n}$ is free.
First, we assume $R$ is $n$-quasi-$F$-split, then there exists $\var{A}$-module homomorphism $\var{\psi} \colon Q_{R,n} \to \var{R}$ such that $\var{\psi}(1)=1$.
Since $Q_{A,n}$ is a free $\var{A}$-module, there exists $\var{A}$-module homomorphism $\psi \colon Q_{A,n} \to \var{A}$ such that the diagram
\[
\begin{tikzcd}
    Q_{A,n} \arrow[r,"\psi"] \arrow[d] & \var{A} \arrow[d] \\
    Q_{R,n} \arrow[r,"\var{\psi}"] & \var{R},
\end{tikzcd}
\]
and in particular, $\psi$ satisfies conditions (a) and (b).
Next, we assume that there exists an $\var{A}$-module homomorphism $\psi \colon Q_{A,n} \to \var{A}$ such that $\psi$ satisfies  conditions (a) and (b).
Thus, the homomorphism $\psi$ induces the $\var{R}$-module homomorphism $\var{\psi} \colon Q_{R,n} \to \var{R}$ such that $\var{\psi}(1)$ is a unit.
Replacing $\var{\psi}$ by $\var{\psi}(1)^{-1} \cdot \var{\psi}$, we obtain a splitting of $\Phi_{R,n} \colon \var{R} \to Q_{R,n}$, as desired.
\end{proof}

\begin{lemma}\label{condition-induce}
We use the notation introduced in \cref{notation:Fedder}.
Let $n$ be an integer with $n \geq 2$.
We take $\varphi \in \Hom_{\var{A}}(Q_{A,n},\var{A})$ and $g_1,\ldots,g_n \in \var{A}$ such that $\varphi=\varphi_{(g_1,\ldots,g_n)}$.
By the isomorphism $Q_{A,n} \simeq F_*Q_{A,n-1} \oplus F_*\var{A}$ in the proof of \cref{decomposition-W}, we obtain the homomorphism $\psi \in \Hom_{\var{A}}(F_*Q_{A,n-1},\var{A})$ and $h \in \var{A}$ such that
\[
\varphi(\alpha)=\psi(F_*\Delta_{W_n}(\alpha))+u(F_*(ha_0))
\]
for every $\alpha=(a_0,\ldots,a_{n-1}) \in Q_{A,n}$.
Furthermore, we assume $\varphi$ satisfies  condition (b) in \cref{lift-splitting}.
\begin{enumerate}
    \item There exist $h_1,\ldots,h_n \in \var{A}$ such that $g_s \equiv f^{p^{n+1-s}-p}h_s \mod I^{[p^{n+1-s}]}$ for every $n \geq s \geq 1$, 
    \item We use $h_1,\ldots,h_n$ in (1).
    If $n=2$ or $p \neq 2$, then we have
    \begin{equation}\label{eq:cond-p-geq-3}
        h_n-u(F_*(h_{n-1}\Delta_1(f^{p-1}))) \in f^{p-1}\var{A}+I^{[p]},
    \end{equation}
    and if $n \geq 3$ and $p=2$,
    then we have
    \begin{equation}\label{eq:cond-p-2}
    h_n-u(F_*(h_{n-1}\Delta_1(f^{p-1})))-u(F_*(f^{p-1}\Delta_1(f^{p-1})u(F_*h_{n-2}))) \in f^{p-1}\var{A}+I^{[p]}.
    \end{equation}
\end{enumerate}
\end{lemma}

\begin{proof}
We prove assertion (1) by induction on $n \geq 2$.
We note that $h_n=g_n$ satisfies the desired condition for $s=n$.
Since $\psi$ coincides with the decomposition
\[
 F_*Q_{A,n-1} \xrightarrow{V} Q_{A,n} \xrightarrow{\varphi} \var{A} 
\]
and $\varphi$ satisfies condition (b) in \cref{lift-splitting}, the homomorphism $\psi$ induces the $\var{R}$-module homomorphism $\psi_R \colon F_*Q_{R,n-1} \to \var{R}$.
Since $\var{R}$ is Gorenstein, the homomorphism $\psi_R$ factors through the homomorphism induced by $u(F_*(f^{p-1}\cdot -)) \colon F_*\var{A} \to \var{A}$.
Therefore, there exists $\varphi' \in \Hom_{\var{A}}(Q_{A,n-1},\var{A})$ such that $\psi$ and $u(F_*(f^{p-1}\cdot -)) \circ \varphi'$ induce same homomorphism $Q_{R,n} \to \var{R}$ and $\varphi'$ satisfies condition (b) in \cref{lift-splitting}.
We take $g'_1,\ldots,g'_{n-1}$ such that $\varphi'=\varphi_{(g'_1,\ldots,g'_{n-1})}$.
By the isomorphism
\[
\Hom_{\var{A}}(F_*Q_{A,n-1},\var{A}) \simeq F^n_*\var{A} \oplus \cdots \oplus F^2_*\var{A},
\]
the homomorphism $\psi$ is corresponding to $(F^n_*g_1,\ldots,F^2_*g_{n-1})$ and $u(F_*(f^{p-1}\cdot -)) \circ \varphi'$ is corresponding to $(F^n_*f^{p^{n-1}(p-1)}g'_1,\ldots,F^2_*f^{p(p-1)}g'_{n-1})$.
Since they induce the same homomorphism, we have $g_s \equiv f^{p^{n-s}(p-1)}g'_s \mod I^{[n+1-s]}$ for every $n-1 \geq s \geq 1$.
For $n=2$, we have $g_1 \equiv f^{p^2-p}g'_1 \mod I^{[p^2]}$, thus $h_1=g'_1$ satisfies the desired condition.
For $n \geq 3$, by induction hypothesis, we have
\[
g_s \equiv f^{p^{n-s}(p-1)}g'_s \equiv f^{p^{n-s}(p-1)}f^{p^{n-s}-p}h_s \mod f^{p^{n-s}(p-1)}I^{[p^{n-s}]} \subseteq I^{[p^{n+1-s}]}
\]
for $n-1 \geq s \geq 1$, as desired.

Next, we prove assertion (2).
We take $r \geq j \geq 1$ and $a \in A$.
Since $\varphi$ satisfies condition (b), we obtain 
\begin{equation}\label{eq:in-I}
    \varphi([af_j]) =\sum_{s=1}^{n} u^{n+1-s}(F^{n+1-s}_*(g_s\Delta_{n-s}(af_j))) \in I.
\end{equation}
where we define $\Delta_0(af_j)=af_j$.
For $n-s \geq 3$ or $p \neq 2$, we have $g_s \Delta_{n-s}(af_j) \in I^{[p^{n+1-s}]}$ by \cref{thm:delta formula} and assertion (1).
Furthermore we have 
\begin{align*}
    g_{n-1}\Delta_1(af_j) &\equiv h_{n-1}f^{p^2-p}\Delta_1(af_j) \mod I^{[p^2]} \\
    &=h_{n-1}\Delta(f^{p-1}af_j)-h_{n-1}a^pf_j^p\Delta_1(f^{p-1}) \\
    &\equiv -h_{n-1}a^pf_j^p\Delta_1(f^{p-1}) \mod I^{[p^2]}.
\end{align*}
Therefore, if $n=2$ or $p \neq 3$, we have
\[
u(F_*( af_j(g_n-u(F_*(h_{n-1}\Delta_1(f^{p-1}))))  )) \in I
\]
for every $a \in \var{A}$ and $r \geq j \geq 1$.
Therefore, we have 
\[
g_n-u(F_*(h_{n-1}\Delta_1(f^{p-1}))) \in (I^{[p]} \colon I)=f^{p-1}\var{A}+I^{[p]}.
\]
On the other hand, if $n \geq 3$, $p=2$ and $s \geq 3$, then we have
\[
g_{n-s}\Delta_s(af_j) \equiv f^{p^{s+1}-p}h_{n-s}(af_j)^{p^s-p^2}\Delta_1(af_j)^p \equiv 0 \mod I^{[p^{s+1}]},
\]
and if $n \geq 3, p=s=2$, then we have
\begin{align*}
    g_{n-2}\Delta_2(af_j) &\equiv f^{p^{3}-p}h_{n-2}\Delta_1(af_j)^p \mod I^{[p^3]} \\
    &\equiv h_{n-2}f^{p-1}(f^{p^2-p}\Delta_1(af_j))^p \mod I^{[p^3]} \\
    &\equiv h_{n-2}f^{p-1}a^{p^2}f_j^{p^2}\Delta_1(f^{p-1})^p \mod I^{[p^2]}.
\end{align*}
Therefore, by the same argument in the case of $n=2$ or $p \neq 3$, we have
\[
h_n-u(F_*(h_{n-1}\Delta_1(f^{p-1})))-u(F_*(f^{p-1}\Delta_1(f^{p-1})u(F_*h_{n-2}))) \in f^{p-1}\var{A}+I^{[p]},
\]
as desired.
\end{proof}

\begin{proof}[Proof of \cref{Fedder}]
First, we note that sequences $\{I_n\}$ and $\{I'_n\}$ are increasing.
Indeed, $I_1 \subseteq I_2$ by definition and if $I_{n-1} \subseteq I_n$, then 
\[
I_n =u(F_*(I_{n-1}\Delta_1(f^{p-1})))+I_1 \subseteq u(F_*(I_{n}\Delta_1(f^{p-1})))+I_1=I_{n+1}
\]
for every $n \geq 2$.
Furthermore, $I'_n \subseteq I'_{n+1}$ for every $n \geq 1$ follows from the same argument.
We prove assertion (1).
First, we consider the following claim.
\begin{claim}\label{cl:fedder1}
Let $g \in A$ and $n \geq 1$.
Then $g \in I_n$ if and only if there exists $g_1,\ldots,g_n \in \var{A}$ such that $\varphi_{(g_1,\ldots,g_n)}$ satisfies condition (b) in \cref{lift-splitting} and $g_n \equiv g \mod I_{n-1}$, where if $n=1$, we regard $I_0$ as $(0)$.
\end{claim}
\begin{claimproof}
We prove the assertion by induction on $n$.
For $n=1$, since $I_1=(I^{[p]} \colon I)$, $g \in I_1$ if and only if $u(F_*(g \cdot -))$ satisfies condition (b) in \cref{lift-splitting}, as desired.
Next, we assume $n \geq 2$.
We take $g \in I_n$, then there exists $g' \in I_{n-1}$ such that $g-u(F_*(g'\Delta_1(f^{p-1}))) \in I_1$.
By the induction hypothesis, there exist  $g'_1,\ldots,g'_{n-1}$ such that $\varphi_{(g_1',\ldots,g'_{n-1})}$ satisfies condition (b) in \cref{lift-splitting}, $g'_{n-1} \equiv g' \mod I_{n-2}$, then by \cref{condition-induce}, there exist $h_1,\ldots,h_n \in \var{A}$ such that $g'_s \equiv f^{p^{n-s}-p}h_s \mod I^{[p^{n-s}]}$ for each $n-1 \geq s \geq 1$.
We define $\varphi'$ by 
\[
\varphi' \colon F_*Q_{A,n-1} \xrightarrow{F_*\varphi_{(g_1',\ldots,g'_{n-1})}} F_*\var{A} \xrightarrow{u(F_*f^{p-1} \cdot-)} \var{A}
\]
and $\varphi$ by
\[
\varphi \colon Q_{A,n} \xrightarrow{\sim} F_*Q_{A,n-1} \oplus F_*\var{A} \xrightarrow{(\varphi',u(F_*g_n\cdot-))} \var{A},
\]
where
\[
g_n:=
\begin{cases}
    u(F_*(g'_{n-1}\Delta_1(f^{p-1}))) & \textup{$n=2$ or $p \neq 2$} \\
    u(F_*(g'_{n-1}\Delta_1(f^{p-1})))+u(F_*(f^{p-1}\Delta_1(f^{p-1})u(F_*h_{n-2}))) & \textup{otherwise}.
\end{cases}
\]
Since we have
\[
u(F_*(g'_{n-1}\Delta_1(f))) \equiv u(F_*(g'\Delta_1(f))) \equiv g \mod I_n
\]
and $u(F_*(f^{p-1}\Delta_1(f^{p-1})u(F_*h_{n-2}))) \in I_2$, we obtain $g_n \equiv g \mod I_n$.
Thus, it is enough to show that $\varphi$ satisfies condition (b) in \cref{lift-splitting}.
We take $g_1,\ldots,g_{n-1}$ such that $\varphi=\varphi_{(g_1,\ldots,g_n)}$, then we have 
\[
g_s=f^{p^{n-s}(p-1)}g'_s \equiv f^{p^{n+1-s}-p}h_s \mod I^{[p^{n+1-s}]}
\]
for $n-1 \geq s \geq 1$.
By the proof of \cref{condition-induce}, for $a \in \var{A}$, $r \geq j \geq 1$, we have
\[
\begin{array}{ll}
    \varphi([af_j]) \equiv u(F_*(af_j(g_n-u(F_*(g'_{n-1}))))) \mod I &  \textup{ if $n=2$ or $p\neq2$}
\end{array}
\]
and
\[
\varphi([af_j]) \equiv u(F_*(af_j(g_n-u(F_*(g'_{n-1}\Delta_1(f^{p-1})))-u(F_*(f^{p-1}\Delta_1(f^{p-1})u(F_*h_{n-2}))) ) )) \mod I
\]
otherwise.
By the choice of $g_n$, we have $\varphi([af_j]) \in I$ for all $a \in \var{A}$ and $r \geq j \geq 1$.
Furthermore, since we have
\[
[\sum_{j=1}^{r}a_jf_j] \equiv \sum_{j=1}^{r}[a_jf_j] \mod V(W_{n-1}((f_1,\ldots,f_r)))
\]
and 
\begin{align*}
    \varphi(V\alpha)=\varphi'(\alpha) \in I
\end{align*}
for $\alpha \in W_{n-1}((f_1,\ldots,f_r))$, we obtain that $\varphi$ satisfies condition (b) in \cref{lift-splitting}, as desired.

Next, we prove the converse direction, thus
we take $g_1,\ldots,g_n \in \var{A}$ such that the corresponding homomorphism $\varphi:=\varphi_{(g_1,\ldots,g_n)}$ satisfies condition (b) in \cref{lift-splitting}, then it is enough to show that $g_n \in I_n$.
For $n=1$, we have $g_1 \in (I^{[p]} \colon I))=I_1$, as desired, thus we assume $n \geq 2$.
By \cref{condition-induce}, there exist $h_1,\ldots,h_n \in \var{A}$ such that $g_s \equiv f^{p^{n+1-s}-p}h_s \mod I^{[p^{n+1-s}]}$, and in particular, $h_n \equiv g_n \mod I^{[p]}$.
By the proof of \cref{condition-induce} and the induction hypothesis, we obtain that $h_{n-1} \in I_{n-1}$, thus we have $u(F_*(h_{n-1}\Delta_1(f^{p-1}))) \in I_n$.
Furthermore, by the construction of $I_2$, we have 
\[
u(F_*(f^{p-1}\Delta_1(f^{p-1})u(F_*h_{n-2}))) \in I_2.
\]
Therefore, by \cref{condition-induce} (2), we have $g_n \in I_n$, as desired.
\end{claimproof}\\
Finally, we prove the assertion.
We have
\begin{align*}
    &\Im(\Hom_{\var{R}}(Q_{R,n}, \var{R}) \xrightarrow{\mathrm{ev}} \var{R}) \\
    \overset{(\star_1)}{=}& \{\varphi(1) \mid \textup{$\varphi \in \Hom_{\var{A}}(Q_{A,n},\var{A})$ satisfying condition (b) in \cref{lift-splitting}}\}\cdot \var{R} \\
    \overset{(\star_2)}{\equiv}& u(F_*I_n)\cdot \var{R} \mod u(F_*I_{n-1})\cdot \var{R},
\end{align*}
where $(\star_1)$ follows from the proof of \cref{lift-splitting} and $(\star_2)$ follows from \cref{cl:fedder1} and $\varphi_{(g_1,\ldots,g_n)}(F_*1)=u(F_*g_n)$.
In conclusion, we have $\sht(R):=\inf\{n \geq 1 \mid I_n \nsubseteq \m^{[p]}\}$, as desired.

Next, we prove assertion (2).
We define the sequence of ideals $\{I^e_n\}$ by $I^e_1:=f^{p-1}u^e(F^e_*f^{p^e-1}\var{A})+I^{[p]}$ and 
\[
I^e_{n}:=u(F_*(I^e_{n-1}\Delta_1(f^{p-1}))+I_1,
\]
inductively.
Then $\{I^e_n\}$ is increasing and $I_n \subseteq I^e_{n+1}$ for every $n,e \geq 1$.
First, we consider the following claim.
\begin{claim}\label{cl:fedder2}
Let $g \in A$ and $n,e \geq 1$.
Then $g \in I^e_n$ if and only if there exists $g_1,\ldots,g_n \in \var{A}$ such that $\varphi_{(g_1,\ldots,g_n)}$ satisfies  condition (b) in \cref{lift-splitting}, there exists $h \in f^{p^{e+n}-1}\var{A}$ with $u^e(F^e_*h) \equiv g_1 \mod I^{[p]}$,  and $g_n \equiv g \mod I^e_{n-1}$, where if $n=1$, we regard $I^e_0$ as $(0)$.
\end{claim}
\begin{claimproof}
For $n=1$, it follows from definition and the fact that  condition (b) in \cref{lift-splitting} is equivalent to $g_1 \in f^{p-1}\var{A}+I^{[p]}$.
We assume $n \geq 2$.
We take $g \in I^e_n$, then there exists $g' \in I^e_{n-1}$ such that $g-u(F_*(g'\Delta_1(f^{p-1}))) \in I_1$.
By induction hypothesis, there exist $g_1',\ldots,g'_{n-1}$ such that $\varphi_{(g'_1,\ldots,g'_{n-1})}$ satisfies condition (b) in \cref{lift-splitting}, we have $u^e(F^e_*h')=g'_1$ for some $h' \in f^{p^{e+n-1}-1}\var{A}$, and $g'_{n-1} \equiv g' \mod I^e_{n-2}$.
By the proof of \cref{cl:fedder1}, if we put $g_s:=f^{p^{n-s}(p-1)}g'_s$ for every $n-1 \geq s \geq 1$ and 
\[
g_n:=
\begin{cases}
    u(F_*(g'_{n-1}\Delta_1(f^{p-1}))) & \textup{$n=2$ or $p \neq 2$} \\
    u(F_*(g'_{n-1}\Delta_1(f^{p-1})))+u(F_*(f^{p-1}\Delta_1(f^{p-1})u(F_*h_{n-2}))) & \textup{otherwise},
\end{cases}
\]
then $\varphi_{(g_1,\ldots,g_n)}$ satisfies condition (b) \cref{lift-splitting} and $g \equiv g_n \mod I'_n$.
We note that $I_2 \subseteq I^e_n$ if $n \geq 3$.
Furthermore, we have
\[
u^e(F^e_*(f^{p^{e+n}-p^{e+n-1}}h')) =f^{p^n-p^{n-1}}u^e(F^e_*h')=f^{p^n-p^{n-1}}g'_1=g_1,
\]
as desired.

Next, we prove the converse direction, thus
we take $g_1,\ldots,g_n \in \var{A}$ such that the corresponding homomorphism $\varphi:=\varphi_{(g_1,\ldots,g_n)}$ satisfies condition (b) in \cref{lift-splitting} and there exists $h \in f^{p^{e+n}-1}\var{A}$ such that $u^e(F^e_*h) \equiv g_1 \mod I^{[p]}$, then it is enough to show that $g_n \in I^e_n$.
For $n=1$, it follows from the definition of $I^e_1$.
For $n \geq 2$, we take $h_1,\ldots,h_n$ as in \cref{condition-induce}, then the assertion follows from $h_{n-1} \in I^e_{n-1}$ by the induction hypothesis and 
\[
u(F_*(f^{p-1}\Delta_1(f^{p-1})u(F_*h_{n-2}))) \in I_2 \subseteq I^e_n
\]
for $n \geq 3$.
\end{claimproof}

If $\sht(R)=\infty$, then $I_n' \subseteq I_n \subseteq \m^{[p]}$ by (1) for every positive integer $n$, thus we may assume $n:=\sht(R) < \infty$.
We note that $I'=u^e(F^e_*f^{p^e-1}\var{A})$ for some $e \geq 1$ by \cite{Gabber}*{Lemma~13.1}, thus if we take a such $e$, then we have $I'_n=I^e_n$.
First, we assume $R$ is quasi-$(F,F^{\infty})$-split.
Then there exist $\varphi \colon Q_{A,n} \to \var{A}$ satisfying conditions (a), (b) in \cref{lift-splitting} and $h \in f^{p^{e+n}-1}\var{A}$ such that we have
\[
u^{e+n}(F^{e+n}_*(ha^{p^{e}}))=\varphi(0,\ldots,0,a)
\]
for every $a \in \var{A}$.
If $\varphi$ corresponds to $(F^n_*g_1,\ldots,F_*g_n)$, then we have
\[
u^{e+n}(F^{e+n}_*(ha^{p^e}))=u^n(F^n_*(g_1a))
\]
for every $a \in \var{A}$, and in particular, we have $u^e(F^e_*h)=g_1$.
By \cref{cl:fedder2}, we have $g_n \in I^e_n=I'_n$.
Since $\varphi(1)=u(F_*g_n) \notin \m$, we have $I'_n \nsubseteq \m^{[p]}$, as desired.
On the other hand, if $I^e_n=I'_n \nsubseteq \m^{[p]}$, then there exists $\varphi=\varphi_{(g_1,\ldots,g_n)}$ such that $\varphi$ satisfies  condition (b) in \cref{lift-splitting}, there exists $h \in f^{p^{e+n}-1}\var{A}$ such that $u^e(F^e_*h) \equiv g_1 \mod I^{[p]}$, and $u(F_*g_n) \notin \m$ by \cref{cl:fedder2}.
Then $(\varphi,u^{e+n}(F^{e+n}_*h\cdot-)) \colon Q_{A,n} \oplus F^{e+n}_*\var{A} \to \var{A}$ induces the homomorphism $Q_{R,n,e} \to \var{R}$ and $(\varphi,u(F^e_*h\cdot-))(1)$ is a unit, as desired.  
\end{proof}

\bibliographystyle{skalpha}
\bibliography{bibliography.bib}

\end{document}